\documentclass[11pt,reqno]{amsart}
\usepackage[width=150mm,height=220mm,
centering,
]{geometry}
\usepackage{xcolor}
\usepackage{marvosym,mathrsfs}
\usepackage{amsmath,amsfonts,amssymb}
\numberwithin{equation}{section}
\allowdisplaybreaks
\usepackage{tikz}
\def\bbR{\mathbb{R}}
\def\bbZ{\mathbb{Z}}

\ifcsname proof\endcsname
\else

\fi

\newtheorem{Theorem}{Theorem}[section]
\newtheorem{Proposition}[Theorem]{Proposition}
\newtheorem{Lemma}[Theorem]{Lemma}
\newtheorem{Example}[Theorem]{Example}
\newtheorem{Definition}[Theorem]{Definition}
\newtheorem{Remark}[Theorem]{Remark}
\newtheorem{Corollary}[Theorem]{Corollary}

\def\rmd{\mathrm{d}}
\begin{document}

\parindent 2em
\title{Nonuniform Sobolev Spaces}

\thanks{T. Chen was supported by the
National Natural Science Foundation of China (No. 12271267)
and the Fundamental Research Funds for the Central Universities.}

\thanks{L. Grafakos was supported by  a
Simons Grant (No. 624733). }

\thanks{W. Sun was supported by the
National Natural Science Foundation of China (No. 12171250 and U21A20426).}

 \author{Ting Chen}
\address{School of Mathematical Sciences and LPMC,
  Nankai University, Tianjin,  China}
\email{t.chen@nankai.edu.cn}

 \author{Loukas Grafakos}
\address{Department of Mathematics, University of Missouri, Columbia, MO 65211, USA}

\email{grafakosl@missouri.edu}

 \author{Wenchang Sun}
\address{School of Mathematical Sciences and LPMC,
  Nankai University, Tianjin,  China}
\email{sunwch@nankai.edu.cn}

\date{}
\begin{abstract}
We study nonuniform Sobolev spaces, i.e., spaces    of functions
whose partial   derivatives   lie in possibly different Lebesgue spaces.
Although standard   proofs do not apply, we show that nonuniform Sobolev spaces share similar
properties as the classical ones.  These spaces arise naturally in the study of certain PDEs.
For instance, we illustrate that
nonuniform fractional Sobolev spaces  are useful in
the study of local estimates for solutions of
heat equations
and the convergence of  Schr\"odinger operators.
In this work we extend recent advances on
local energy estimates for solutions of heat equations
and the
convergence of Schr\"odinger operators
to  nonuniform fractional Sobolev spaces.
\end{abstract}


\maketitle

\section{Introduction}

Given a  positive integer $m$,
a positive number $p\in[1,\infty]$
and an open set $\Omega\subset \bbR^N$,
we denote by
\[
W^{m,p}(\Omega) = \{f\in L^p(\Omega):\,  D^{\alpha}  f\in L^p(\Omega), |\alpha|\le m\}  ,
\]
the classical Sobolev space, where
$\alpha =(\alpha_1,\ldots,\alpha_N)$ is a multi-index,
$|\alpha| = \sum_{i=1}^N \alpha_i$,
and
$D^{\alpha} $ denotes the  weak partial derivatives,
that is, partial derivatives in the sense of distributions.
The theory of Sobolev spaces plays an import role in the study of
partial differential equations and many other fields.
We refer to  text books by
Adams and Fournier~\cite{AdamsFournier2003},
Demengel and Demengel~\cite{Demengel2012},
Evans~\cite{Evans2010},
Grafakos~\cite{Grafakos2014,Grafakos2014m},
Leoni~\cite{Leoni2017},
Pi\c{s}kin and Okutmu\c{s}tur~\cite{Pi2021}
for an overview of Sobolev spaces and applications in PDEs
and in harmonic analysis.

Since a function and its derivatives might have different properties,
it is not necessary for them to lie in the same Lebesgue space.
In this paper, we consider {\em nonuniform Sobolev    spaces} on $\bbR^N$, i.e., spaces
for which a function and its derivatives belong
to different Lebesgue spaces.

\begin{Definition}
Let $\Omega\subset\bbR^N$ be an open set.
For $k\ge 1$ and $\vec p = (p_0, \ldots,p_k)\in [1,\infty]^{k+1}$,
the nonuniform Sobolev space
$W_k^{\vec p}(\Omega)$ consists of all measurable functions $f$
for which  $\partial^{\alpha}f\in L^{p_{|\alpha|}}(\Omega)$, where $|\alpha|\le k$.
For $f\in W_k^{\vec p}$, define its norm by
\[
  \|f\|_{W_k^{\vec p}} = \sum_{|\alpha|\le k}  \|  D^{\alpha } f\|_{L^{p_{|\alpha|}}}.
\]
\end{Definition}

These spaces naturally arise in the study of certain PDEs.
In particular, some equations have solutions in
nonuniform Sobolev spaces but have no solution
in the classical Sobolev spaces.
For example, consider positive solutions
of the critical $p$-Laplace equation
\begin{equation}\label{eq:p}
 \Delta_p u  + u^{p^*-1}=0
\end{equation}
in $\bbR^N$, where $1<p<N$, $p^* = Np/(N-p)$
and $\Delta_p u = \mathrm{div} (|\nabla u|^{p-2} \nabla u)$.

This equation is well studied in the literature.
It was shown
by Damascelli, Merch\'{a}n, Montoro and Sciunzi \cite{DamascelliMerchan2014},
Sciunzi \cite{Sciunzi2016} and V\'{e}tois \cite{Vetois2016}
that
if a function $u$ in the class
\[
  \mathcal D^{1,p}(\bbR^N) = \{u\in L^{p^*}(\bbR^N):\, \nabla u\in L^p(\bbR^N)\}
\]
is a solution of (\ref{eq:p}), then it is of the form
\[
  u(x) =
  U_{\lambda,x_0} (x) = \left(
  \frac{\lambda^{1/(p-1)} N^{1/p}((N-p)/(p-1))^{{(p-1)}/p}}
  {\lambda^{p/(p-1)}+|x-x_0|^{p/(p-1)}}
  \right)^{(N-p)/p},
\]
where $\lambda$ is a positive constant
and $x_0$ is a point in $\bbR^N$.
Moreover, Catino, Monticelli and Roncoroni \cite{CatinoMonticelliRoncoroni2023}
provided the classification of positive solutions to the critical
$p$-Laplace equation.
See also Ciraolo, Figalli and Roncoroni \cite{CiraoloFigalliRoncoroni2020}
for solutions  in convex cones.

Observe that $\mathcal D^{1,p}$ itself is
the nonuniform Sobolev space $W_1^{(p^*,p)}$.
Since $|U_{\lambda,x_0}(x)|\approx |x|^{-(N-p)/(p-1)}$
whenever  $|x|$ is large enough,
it is easy to see that (\ref{eq:p}) has a
positive  solution
in $W^{1,p}$ if and only if $1<p<N^{1/2}$.
However, we can increase the integrability index $p$ of the solution if we consider
nonuniform Sobolev spaces. Indeed, we obtain that
(\ref{eq:p}) has a positive solution
in $W_1^{(p_0,p)}$ whenever $1<p<N$ and $p_0 > N(p-1)/(N-p)$.

On the other hand, for  $1\le p_1< N$ and $f$ in the classical Sobolev space
$W^{1,p_1}$,
the Sobolev inequality says that
\begin{equation}\label{eq:e10}
  \|f\|_{L^{Np_1/(N-p_1)}}  \le C_{N,p_1} \|\nabla f\|_{L^{p_1}}.
\end{equation}
This is proved by showing that
\begin{equation}\label{eq:e11}
   \|f\|_{L^{Np_1/(N-p_1)}}^{(N-1)p_1/(N-p_1)} \le C_{N,p_1} \|\nabla f\|_{L^{p_1}}
    \|f\|_{L^{Np_1/(N-p_1)}}^{N(p_1-1)/(N-p_1)}
\end{equation}
for   compactly supported differentiable functions
$f$.
We refer to \cite{Demengel2012} for details.

For the nonuniform case, that is, for  $f\in C^1\cap W_k^{\vec p}$ with $\vec p =(p_0,p_1)$,
both (\ref{eq:e10}) and (\ref{eq:e11}) are still true.
However, their proofs are quite different
from the  uniform case.
In the classical uniform case, one first obtains the density of compactly supported differentiable functions,
then   derives (\ref{eq:e10}) from  (\ref{eq:e11})
for such functions, and finally  extends (\ref{eq:e10}) to
  all functions in $W^{1,p_1}$ by   density.
For the nonuniform case, compactly supported differentiable functions are still dense in $W_1^{\vec p}$. However,
the embedding inequality is required in the proof.
So we have to prove  (\ref{eq:e10})
directly for functions which is not compactly supported.
In this case,  (\ref{eq:e10}) is not a straightforward consequence of (\ref{eq:e11}):
we have to show first that $f\in L^{Np_1/(N-p_1)}$.

In some applications, only the   derivatives
are concerned. For example,
Fefferman, Israel and Luli
\cite{FeffermanIsraelLuli2014b,FeffermanIsraelLuli2014}
studied Sobolev extension operators
for homogeneous Sobolev spaces.
In this case, it is natural to consider nonuniform
Sobolev spaces.

Another application we discuss concerns   local estimates for
solutions of   heat equations.
We obtain   local energy estimates
for initial data in nonuniform Sobolev spaces; this
  extends the local estimates
by Fefferman,  McCormick, Robinson and Rodrigo \cite{FeffermanMcCormickRobinsonRodrigo2017}.

Moreover, nonuniform Sobolev spaces are also
useful in the  study of the
 convergence of Schr\"odinger operators defined by
\[
  e^{it(-\Delta)^{a/2}}f(x):=  \frac{1}{(2\pi)^N}\int_{\bbR^N} e^{i(x\cdot\omega + t |\omega|^a )}
     \hat f(\omega) \rmd \omega,
\]
where $a> 1$ is a constant.
It is well known that for $f$ nice enough, $e^{it(-\Delta)^{a/2}}f(x)$
is the solution of the fractional Schr\"odinger equation
\[
  \begin{cases}
  i \partial_t u  + (-\Delta_x)^{a/2}u  = 0,  & (x,t)\in \bbR^N\times R,\\
  u(x,0) =f(x), & x\in\bbR^N.
  \end{cases}
\]

For the case $N=1$ and $a=2$,
Carleson~\cite{Carleson1980} studied the convergence of $e^{-it\Delta}f(x)$ as $t$ tends to $0$ for
functions $f$ in the Sobolev space
\[
  H^s(\bbR^N):= \{f\in L^2:\, (1+|\omega|^2)^{s/2} \hat f(\omega) \in L^2   \}.
\]
It was shown \cite{Carleson1980} that when $a=2$,
\begin{equation}\label{eq:s2:e1}
  \lim_{t\rightarrow 0}e^{it(-\Delta)^{a/2}}f(x) = f(x),\quad a.e.
\end{equation}
for all $f\in H^s(\bbR)$ if $s\ge 1/4$.

Since then, many works have appeared on this topic.
Dahlberg and Kenig \cite{DahlbergKenig1982} provided counterexamples  indicating that
the range $s\ge 1/4$ is sharp for $N=1$.
And Sj\"olin \cite{Sjolin1987}
extended this result to the   case $a>1$.

For $a=2$ and higher dimensions $N\ge 2$,
Sj\"olin \cite{Sjolin1987}  and
 Vega \cite{Vega1988}
 proved that (\ref{eq:s2:e1}) is true when $s>1/2$.
For the case $N=2$, this result was improved to
$s>3/8$ by Lee \cite{Lee2006}.
Bourgain ~\cite{Bourgain2013,Bourgain2016} proved that
$s> 1/2 - 1/(4N)$ is sufficient
and
$s\ge 1/2 - 1/(2N+2)$ is necessary
for the convergence.
Du, Guth and Li \cite{DuGuthLi2017} proved that
$s> 1/2 - 1/(2N+2)$ is sufficient for the dimension $N=2$
and Du and Zhang \cite{DuZhang2019} showed that it is also true
for general
$N\ge 3$.

For the general case $a>1$ and $N\ge 2$,
Sj\"olin \cite{Sjolin1987}
 proved that (\ref{eq:s2:e1}) is valid when $s>1/2$.
Prestini \cite{Prestini1990} showed that
$s\ge 1/4$ is necessary and sufficient for radial functions $f$
and dimensions $N\ge 2$.
Cho and Ko \cite{ChoKo2022} proved that
$s>1/3$ is sufficient  for
the dimension $N=2$.

Related works also include the
non-tangential convergence by Shiraki \cite{Shiraki2020}, Yuan, Zhao
 and Zheng \cite{YuanZhaoZheng2021}, Li, Wang and Yan \cite{LiWang2021,LiWangYan2021},
the convergence along curves by Cao and Miao \cite{CaoMiao2023} and Zheng \cite{Zheng2019},
the Hausdorff dimension of
the divergence set by Li, Li and Xiao \cite{LiLiXiao2021},
and the convergence rate by Cao, Fan and Wang \cite{CaoFanWang2018};
see also the works by Cowling \cite{Cowling1983},
Walther \cite{Walther1995}, and Rogers and Villarroya \cite{RogersVillarroya2008}
 for the case $a<1$.

Although the  range   $s\ge 1/4$ of the index $s$  is sharp
 for $N=1$,
and $s> 1/2 - 1/(2N+2)$ is also sharp up to the endpoint
for $N\ge 2$,
we show that the result can be further extended when
nonuniform Sobolev spaces are considered.

Suppose that $0<s<1$ and $1\le p<\infty$.
Recall that the  classical fractional Sobolev space   $W_s^p$
consists of all measurable functions $f$ for which
\[
  \|f\|_{W_s^p} :=  \|f\|_{L^p}  +  [f]_{W_s^p}<\infty,
\]
where
\[
  [f]_{W_s^p} :=
   \bigg(\iint_{\bbR^N\times \bbR^N}
  \frac{|f(x)-f(y)|^{p}}{|x-y|^{N+sp}} \rmd x\rmd y
 \bigg)^{1/p}
\]
is the Gagliardo seminorm of $f$.
Fractional Sobolev spaces
were introduced by Aronszajn \cite{Aronszajn1955}, Gagliardo  \cite{Gagliardo1958}
and  Slobodecki\u{\i} \cite{Slobodeckij1958}.
We refer to the text book by Leoni~\cite{Leoni2023}
and  papers by
Brezis and Mironescu \cite{BrezisMironescu2001}
and Di Nezza~\cite{DiNezza2012} for an overview of
fractional Sobolev spaces.
See also recent papers
by Brezis,  Van Schaftingen  and   Yung
\cite{BrezisVanSchaftingenYung2021}
and
Gu and Yung
\cite{GuYung2021}
for the limit of
norms of fractional Sobolev spaces.

We now provide a formal definition of nonuniform  fractional Sobolev spaces.
\begin{Definition}\label{def:D1}
Given $0<s<1$ and  $\vec p = (p_0, p_1)\in [1,\infty)^2$,
the nonuniform  fractional Sobolev space is defined by
\begin{align*}
W_s^{\vec p}(\bbR^N) &= \{f:\,  \|f\|_{W_s^{\vec p}}
   :=\|f\|_{L^{p_0}} + [f]_{W_s^{p_1}} <\infty\}.
\end{align*}

For the case $s>1$, set $\nu_s = s - \lfloor s\rfloor$
and
 $\vec p = (p_0,\ldots, p_{\lceil s\rceil})$.
If $s$ is not an integer, define
\[
  W_s^{\vec p}(\bbR^N) = \Big \{f:\, \|f\|_{W_s^{\vec p}}:=
    \sum_{l=0}^{\lfloor s\rfloor}  \sum_{|\alpha|=l}\|D^{\alpha}f\|_{L^{p_l}}
     +  \sum_{|\alpha|=\lfloor s\rfloor}
      [D^{\alpha}f]_{W_{\nu_s}^{p_{\lceil s\rceil}}}<\infty\Big\}.
\]
Here $\lfloor s\rfloor$ stands for the greatest integer which is less than or equal to $s$,
and $\lceil s\rceil$ stands for the ceiling, i.e., the least
 integer which is greater than or equal to $s$.
\end{Definition}

Recall that if $p_0=\ldots = p_{\lceil s\rceil}=2$,
then $W_s^{(2,\ldots,2)}(\bbR^N)$ coincides with $ H^s(\bbR^N)$.
We refer to \cite[Proposition 4.17]{Demengel2012} for a proof.

The introduction and study of these spaces is   motivated by the fact  that they   provide
stronger results than classical Sobolev spaces. As an application
we show that if the convergence of  Schr\"odinger operators
holds for all functions in $H^s$ for some $s>0$,
then it also holds for all functions in $W_s^{\vec p}$
with the same index $s$ if
$p_{\lceil s\rceil}=2$.
Moreover, we also obtain convergence results in the
case $1<p_{\lceil s\rceil}< 2$.

The paper is organized as follows.
In Section 2, we show that   compactly supported
infinitely differentiable functions are dense in
nonuniform Sobolev spaces.
We also obtain the Sobolev inequality
for functions in nonuniform Sobolev spaces.
In Section 3, we focus on nonuniform fractional
Sobolev spaces and present an embedding theorem
for such spaces.
And in Section 4, we give two applications of
nonuniform Sobolev spaces.
We give local energy estimates for solutions
of the heat equation with initial data
in nonuniform Sobolev spaces.
And we  prove the almost everywhere convergence
of Schr\"odinger operators for a large class
of functions,   extending known results.

\textbf{Symbols and Notations}. \,
$\{e_i:\, 1\le i\le N\}$ stands for the canonical  basis for $\bbR^N$,
that is, $e_i = (0,\ldots,1,\ldots,0)\in\bbR^N$, where
only the $i$-th component is $1$ and all others are $0$.
For any $x=(x_1,\ldots,x_N)\in\bbR^N$, we have   $x= \sum_{i=1}^N x_ie_i$.
$B(0,R)$ stands for the ball $\{x\in\bbR^N:\, |x|<R\}$.

For a tempered distribution $f$ and a function $\varphi$
in the Schwartz class $\mathscr S$, the notation
$
   \langle f, \varphi\rangle
$
stands for the value of the action of $f$   on $\varphi$.

\section{Embedding Inequalities for Nonuniform
Sobolev Spaces}

Denote by $C_c^{\infty}(\bbR^N)$
the function space consisting  of all  compactly supported
infinitely
differentiable functions.
In this section we obtain
the density of $C_c^{\infty}(\bbR^N)$  in nonuniform Sobolev spaces.
We then obtain a Sobolev embedding theorem for nonuniform Sobolev spaces;
and the aforementioned density  is a  crucial tool in this embedding.
Both of these results seem to require   proofs that are quite different from the classical case.
For the nonuniform case, we need first to prove the
density for special indices. Then we
deduce an embedding result, with which we finally obtain
the density of $C_c^{\infty}(\bbR^N)$ for full indices.

We begin with a simple lemma.

\begin{Lemma}\label{Lm:density}
For any $k\ge 1$ and $\vec p\in [1,\infty)^{k+1}$,
 the space $C^{\infty}\cap W_k^{\vec p}(\bbR^N)$ is dense in $W_k^{\vec p}(\bbR^N)$.
\end{Lemma}

\begin{proof}
Take some $f\in  W_k^{\vec p}$ and
$\varphi\in C_c^{\infty}$ such that
$\varphi$ is nonnegative and $\|\varphi\|_{L^1}=1$. Set
\[
  g_{\lambda} (x) = \int_{\bbR^N} f(y) \frac{1}{\lambda^N}\varphi(\frac{x-y}{\lambda})
    \rmd y, \qquad \lambda>0.
\]
We have $g_{\lambda}\in C^{\infty}$.
For any multi-index $\alpha$ with
$|\alpha|\le k$,  we have
\begin{align}
  \partial^{\alpha} g_{\lambda} (x)
   &= \int_{\bbR^N} \partial^{\alpha}f(y) \frac{1}{\lambda^N}\varphi(\frac{x-y}{\lambda})
    \rmd y   \label{eq:e1}\\
   &= \int_{\bbR^N} \partial^{\alpha}f(x-\lambda y) \varphi(y)
    \rmd y.   \nonumber
\end{align}
It follows from Minkowski's inequality and Lebesgue's dominated convergence
theorem that
\[
  \lim_{\lambda\rightarrow 0} \|\partial^{\alpha} g_{\lambda}
  -
   \partial^{\alpha} f \|_{L^{p_{|\alpha|}}}
\le
     \lim_{\lambda\rightarrow 0}
   \int_{\bbR^N}\| \partial^{\alpha}f(\cdot -\lambda y)
     - \partial^{\alpha}f\|_{L^{p_{|\alpha|}}}
      \varphi(y)
    \rmd y
   =0.
\]
This completes the proof.
\end{proof}

The proof of the density of $C_c^{\infty}$ in nonuniform Sobolev spaces
is split in two cases.
First, we consider special indices.

\begin{Lemma}\label{Lm:density:a}
Suppose that  $k\ge 1$ and $\vec p = (p_0,\ldots,p_k)$
with  $1/p_i \le 1/p_{i-1} + 1/N$, $1\le i\le k$.
Then the space $C_c^{\infty}(\bbR^N)$ is dense in $W_k^{\vec p}(\bbR^N)$.
\end{Lemma}

\begin{proof}
By Lemma~\ref{Lm:density}, for any
$f\in W_k^{\vec p}$ and $\varepsilon>0$,
there is some $g\in C^{\infty}\cap W_k^{\vec p}$ such that
\[
  \|g - f\|_{W_k^{\vec p}} < \varepsilon.
\]
Moreover, if follows  from (\ref{eq:e1}) and Young's inequality that we may choose
the function $g$ such that
\begin{equation}\label{eq:e2}
\partial^{\gamma}g \in L^r ,\qquad \forall\,\,\, r\ge p_{|\gamma|}.
\end{equation}

The proof will be complete if we can  show
that there is some $\tilde g\in C_c^{\infty}$ such that
\begin{equation}\label{eq:e3}
 \| g -\tilde g \|_{W_k^{\vec p}} < C'\varepsilon.
\end{equation}

Take some $\psi\in C_c^{\infty}$ such that
$\psi(x) =1$ whenever $|x|<1$.
Set
  $g_n(x) = \psi(x/n) g (x)$.
It remains to show that the sequence $\{g_n:\, n\ge 1\}$ converges
to $g $ in $W_k^{\vec p}$.

We see from the choice of $\psi$
that $\{g_n:\, n\ge 1\}$ converges
to $g $ in $L^{p_0}$.
On the other hand,
fix some multi-index $\alpha$ with $1\le |\alpha|\le k$.
Then $\partial^{\alpha} g_n(x)$ is the sum of $\psi(x/n) \partial^{\alpha} g (x)$
and terms like $(1/n^{|\beta|})\partial^{\beta}\psi(x/n)
  \partial^{\gamma} g (x)$, where $|\beta| + |\gamma|=|\alpha|$
and $|\beta |\ge 1$.

If $p_{|\alpha|}\ge p_{|\gamma|} $, we deduce from (\ref{eq:e2}) that
\begin{align*}
\Big\|\frac{1}{n^{|\beta|}} \partial^{\beta}\psi(\frac{\cdot }{n})
  \partial^{\gamma} g\Big\|_{L^{p_{|\alpha|}}}
&\le   \frac{1}{n^{|\beta|}} \cdot \| \partial^{\beta}\psi\|_{L^{\infty}}
   \|\partial^{\gamma} g 1_{\{|x|\ge n\}}\|_{L^{p_{|\alpha|}}}
   \\
&\rightarrow  0,  \qquad \mathrm{as}\,\, n\rightarrow \infty.
\end{align*}

If $p_{|\alpha|}< p_{|\gamma|} $, there is some $r>1$ such that
\[
  \frac{1}{p_{|\alpha|}} = \frac{1}{p_{|\gamma|}} + \frac{1}{r}.
\]
Applying H\"older's inequality, we have
\begin{align}
\Big\|\frac{1}{n^{|\beta|}} \partial^{\beta}\psi(\frac{\cdot }{n})
  \partial^{\gamma} g\Big\|_{L^{p_{|\alpha|}}}
&\le   \frac{1}{n^{|\beta|}} \cdot \| \partial^{\beta}\psi(\frac{\cdot }{n})\|_{L^{r}}
   \|\partial^{\gamma} g 1_{\{|x|\ge n\}}\|_{L^{p_{|\gamma|}}}
   \nonumber \\
&= \frac{1}{n^{|\beta|-N/r}}
\cdot \| \partial^{\beta}\psi \|_{L^{r}}
\|\partial^{\gamma} g 1_{\{|x|\ge n\}}
  \|_{L^{p_{|\gamma|}}}.\label{eq:e4}
\end{align}
Since $1/p_i \le 1/p_{i-1} + 1/N$ for all $1\le i\le k$,
we have
\[
  \frac{1}{r} =
  \frac{1}{p_{|\alpha|}} - \frac{1}{p_{|\gamma|}}
\le \frac{|\alpha|-|\gamma|}{N}
=\frac{|\beta|}{N}.
\]
Consequently, $ | \beta |- N/r\ge 0$.
It follows from (\ref{eq:e4})  and \eqref{eq:e2} that
\begin{align}
\Big\|\frac{1}{n^{|\beta|}} \partial^{\beta}\psi(\frac{\cdot }{n})
  \partial^{\gamma} g\Big\|_{L^{p_{|\alpha|}}}
 \rightarrow  0,  \qquad \mathrm{as}\,\, n\rightarrow \infty. \label{eq:e8}
\end{align}
Hence the sequence $\{g_n:\, n\ge 1\}$ converges
to $g $ in $W_k^{\vec p}$ and
(\ref{eq:e3}) is  valid.
\end{proof}

To prove the density of $C_c^{\infty}$ for full indices,
we   first   establish
an embedding theorem.
For classical Sobolev spaces, the embedding inequality
\[
  \|\varphi\|_{L^{N/(N-1)}} \le  C \|\nabla \varphi\|_{L^1}
\]
is first proved for functions in
$C_c^{\infty}$, and then for general functions in $W^{1,1}$ by the
density of $C_c^{\infty}$ in $W^{1,1}$.
We refer to \cite[eq. (2.38)]{Demengel2012} for details.

In the nonuniform case, we have no such  density result
at the moment. So we need to prove it directly
for functions in $C^{\infty}\cap L^{p_0}(\bbR^N) $
for some $p_0>0$.

\begin{Lemma}  \label{Lm:L2}
Suppose that $N\ge 2$.

\begin{enumerate}
\item For any $\varphi\in C^{\infty}\cap L^{p_0}(\bbR^N)$ with $0<p_0<\infty$,
\begin{equation}\label{eq:e6}
  \|\varphi\|_{L^{N/(N-1)}} \le
    \frac{1}{N} \sum_{|\alpha|=1} \|\partial^{\alpha} \varphi\|_{L^1}.
\end{equation}
Moreover, the above inequality is also true for
any $f\in W_1^{\vec p}$ with $\vec p=(p_0,1)$ and $1\le p_0<\infty$.

\item
Suppose that $\vec p =(p_0,p_1)$, where $1\le p_0<\infty$ and $1\le p_1<N$.
For any $f\in W_1^{\vec p}(\bbR^N)$,
we  have
\begin{equation}\label{eq:Np1}
\|f\|_{L^{Np_1/(N-p_1)}}  \le C_{N,p_1} \|\nabla f\|_{L^{p_1}}.
\end{equation}
Hence for all $q$ between $p_0$ and $Np_1/(N-p_1)$, we have
\begin{equation}\label{eq:e13}
\|f\|_{L^q} \le C \|f\|_{W_1^{\vec p}}.
\end{equation}
\end{enumerate}
\end{Lemma}

\begin{proof}
(i)\,\, The proof is similar to that for
 \cite[eq. (2.38)]{Demengel2012}.
 Here we  only  provide a sketch with emphasis on the difference.

Fix some $\varphi\in C^{\infty}\cap L^{p_0} $.
If the right-hand side of (\ref{eq:e6})  equals infinity,
then (\ref{eq:e6}) is certainly valid. So we only need to consider the
case $\|\partial^{\alpha} \varphi\|_{L^1}<\infty$ for all
multi-indices $\alpha$ with $|\alpha|=1$.

For each index $i $, denote
$\tilde x_i = (x_1,\ldots, x_{i-1},
x_{i+1},\ldots, x_N)$. Since $\varphi\in L^{p_0}$, we have
\[
  \int_{\bbR^{N-1}}
  \bigg(  \int_{\bbR} |\varphi(x)|^{p_0}  \rmd x_i \bigg)\,   \rmd \tilde x_i
   = \|\varphi\|_{L^{p_0}}^{p_0}
  <\infty.
\]
Hence there are measurable sets $A_i\subset\bbR^{N-1}$
of measure zero such that
\[
  \int_{\bbR} |\varphi(x)|^{p_0}\rmd x_i<\infty,
  \qquad \forall \,\, \tilde x_i\in \bbR^{N-1}\setminus A_i.
\]
Consequently,  for each $\tilde x_i$
there is a sequence $\{a_n:\, n\ge 1\}\subset \bbR$
 (depending on $\tilde x_i$)
such that $\lim_{n\rightarrow\infty} a_n = -\infty$ and
$\lim_{n\rightarrow\infty} \varphi(x+(a_n-x_i)e_i) = 0$.
Since  for each $ x_i$ we have
\[
  \varphi(x) - \varphi(x+(a_n-x_i)e_i) = \int_{a_n}^{x_i}
    \partial_{x_i}\varphi(x+ (t-x_i)e_i) \rmd t,
\]
   by the Fundamental Theorem of Calculus,  letting $n\rightarrow\infty$, we obtain
\[
  \varphi(x)  = \int_{-\infty}^{x_i} \partial_{x_i}
  \varphi(x+ (t-x_i)e_i) \rmd t,\qquad \forall \, x_i\in\bbR, \,\,
 \forall \,  \tilde x_i\in\bbR^{N-1}\setminus A_i.
\]
 Here and henceforth $e_i=(0,\ldots, 1,\ldots, 0)$ with $1$ only on the $i$th  entry and $0$ elsewhere.
It follows that
\begin{equation}\label{eq:e5}
  |\varphi(x)| \le  \int_{\bbR} \left|\partial_{x_i}
  \varphi(x+ (t-x_i)e_i) \right|\rmd t,\qquad \forall \, x\in \bbR^N\setminus B_i,
\end{equation}
where
\[
  B_i =  \{x\in\bbR^N:\, x_i\in\bbR, \tilde x_i\in A_i\}
\]
is of measure zero in  $\bbR^N$.
Observe that the integral in (\ref{eq:e5})
is independent of $x_i$.

For $1\le i\le N$, define the function $F_i$ on $\bbR^{N-1}$ by
\[
  F_i(\tilde x_i) = \int_{\bbR} \left|\partial_{x_i}
  \varphi(x+ (t-x_i)e_i) \right|\rmd t.
\]
We have
\[
  |\varphi(x)|^{N/(N-1)}
  \le \prod_{i=1}^N F_i(\tilde x_i)^{1/(N-1)},
  \qquad
  \forall x\in\bbR^N \setminus \bigcup_{i=1}^N B_i.
\]
Now following the same arguments as that in \cite[Page 76]{Demengel2012}
we obtain
\[
  \|\varphi\|_{L^{N/(N-1)}} \le
    \frac{1}{N} \sum_{|\alpha|=1} \|\partial^{\alpha} \varphi\|_{L^1}.
\]

For the general case, i.e., $f\in W_1^{\vec p}$ with $p_1=1$,
since $C^{\infty}\cap W_1^{\vec p}$ is dense in
$W_1^{\vec p}$, there is a sequence
$\{\varphi_n:\, n\ge 1\}\subset C^{\infty}\cap W_1^{\vec p}$
which is convergent to $f$ in $W_1^{\vec p}$,
as well as on $\bbR^N$ almost everywhere.

For each $n$, we have
\[
  \|\varphi_n\|_{L^{N/(N-1)}} \le
    \frac{1}{N} \sum_{|\alpha|=1} \|\partial^{\alpha} \varphi_n\|_{L^1}.
\]
Letting $n\rightarrow\infty$, we apply
Fatou's lemma to finally deduce
\[
  \|f\|_{L^{N/(N-1)}} \le
    \frac{1}{N} \sum_{|\alpha|=1} \|\partial^{\alpha} f\|_{L^1}.
\]

\noindent
(ii)\,\,
We see from the embedding theorem
 for  Sobolev's spaces (see \cite[eq. (2.46)]{Demengel2012}) that for any $f\in C_c^{\infty}(\bbR^N)$,
\[
  \|f\|_{L^{Np_1/(N-p_1)}} \le C \|\nabla f\|_{L^{p_1}}.
\]
If $1/p_1 \le 1/p_0 + 1/N$, then $C_c^{\infty}$ is dense in $W_1^{\vec p}$,
thanks to Lemma~\ref{Lm:density:a}.
Similar arguments as the previous case we get that the above inequality is true
for all $f\in W_1^{\vec p}$.

It remains to consider the case $1/p_1 > 1/p_0 + 1/N$.

Take   $f\in W_1^{\vec p}$.
Set $\vec{\tilde p} = (\tilde p_0, p_1)$, where  $1/\tilde p_0 = 1/p_1 - 1/N$.
It is easy to see that $1<\tilde p_0 < p_0$.
We conclude that $f\in W_1^{\vec{\tilde p}}$,
for which we only need to
show that $f\in L^{\tilde p_0}$.

To this end, set $r=p_0/p'_1+1$ and $h = |f|^{r-1} f$.
We have $\nabla h = r |f|^{r-1} \nabla f$. Since
$f\in L^{p_0}$, $\nabla f\in L^{p_1}$ and
$(r-1)/p_0 + 1/p_1=1$, we see from H\"older's
inequality that
\[
  \|\nabla h\|_1 \le r \| |f|^{r-1}\|_{L^{p_0/(r-1)}} \|\nabla f\|_{L^{p_1}}
   = r \| f\|_{L^{p_0}}^{r-1} \|\nabla f\|_{L^{p_1}}
  .
\]
It follows from (\ref{eq:e6}) that
\begin{align*}
 \|h\|_{L^{N/(N-1)}}
 &\le C \|\nabla h\|_{L^1} \le Cr \| f\|_{L^{p_0}}^{r-1} \|\nabla f\|_{L^{p_1}}.
\end{align*}
Denote $q = rN/(N-1)$.
Observe that
$\|h\|_{L^{N/(N-1)}} = \|f\|_{L^q}^r$. We have
\begin{equation}\label{eq:e12}
 \|f\|_{L^q}
 \le  \| f\|_{L^{p_0}}^{1-1/r}  \Big(Cr  \|\nabla f\|_{L^{p_1}}\Big)^{1/r}.
\end{equation}
Since $1/p_1 > 1/p_0 + 1/N$, we have
\[
  \frac{q}{p_0} = \frac{N}{N-1} \Big (\frac{1}{p_0}+1 - \frac{1}{p_1}\Big)< 1.
\]
On the other hand, it follows from $1/p_0 < 1/p_1 - 1/N$ that
\[
  \frac{1}{p'_1} < \frac{p_0}{p'_1} \Big( \frac{1}{p_1} -  \frac{1}{N} \Big) .
\]
Hence,
\[
1- \frac{1}{N}< \frac{1}{p_1} - \frac{1}{N} + \frac{p_0}{p'_1} \Big( \frac{1}{p_1} -  \frac{1}{N} \Big) .
\]
Therefore  we have
\[
\frac{1}{q}=  \frac{N-1}{N}\cdot \frac{1}{1+p_0/p'_1} <  \frac{1}{p_1} -  \frac{1}{N} .
\]

Set $q_0 = p_0$. For $n\ge 1$, define $r_n$ and $q_n$ recursively by
\[
  r_n = \frac{q_{n-1}}{p'_1} + 1 \mbox{\quad and \quad }
  q_n = \frac{r_nN}{N-1}.
\]
We see from above arguments that both $\{q_n:\, n\ge 1\}$
and $\{r_n:\, n\ge 1\}$ are
 decreasing, $1/p_1 > 1/q_n+1/N$ and
\begin{equation}\label{eq:e7}
 \|f\|_{L^{q_n}}
 \le  \| f\|_{L^{q_{n-1}}}^{1-1/r_n}  \Big(Cr_n  \|\nabla f\|_{L^{p_1}}\Big)^{1/r_n}.
\end{equation}
Since $q_n$  is   bounded from below,  the limits   $\tilde r := \lim_{n\rightarrow\infty} r_n$
and  $\tilde q := \lim_{n\rightarrow\infty} q_n$ exist.
Moreover,  $\tilde q =\tilde  rN/(N-1) > \tilde r\ge 1$.

Take some constant $\varepsilon$ small enough such that $\varepsilon\cdot  Cr_1   \|\nabla f\|_{L^{p_1}}<1$.
Since $r_n\le r_1$,
we have $\varepsilon\cdot  Cr_n   \|\nabla f\|_{L^{p_1}}<1$ for all $n\ge 1$.
Set $\tilde f = \varepsilon f$ and substitute $\tilde f$ for $f$ in (\ref{eq:e7}),  we get
\[
 \|\tilde f\|_{L^{q_n}}
 \le  \| \tilde f\|_{L^{q_{n-1}}}^{1-1/r_n},\qquad n\ge 1.
\]
Applying the above inequality recursively, we obtain
\begin{align*}
 \|\tilde f\|_{L^{q_n}}
 &\le  \| \tilde f\|_{L^{q_{n-2}}}^{(1-1/r_n)(1-1/r_{n-1})} \\
 &\le \cdots \\
 &\le \| \tilde f\|_{L^{q_0}}^{\prod_{i=1}^n (1-1/r_i)} \\
 &\le \max\{\| \tilde f\|_{L^{q_0}},1\}.
\end{align*}
Hence,
\[
  \int_{\bbR^N} |\tilde f(x)|^{q_n} \rmd x \le \max\{\| \tilde f\|_{L^{q_0}},1\}^{q_n}.
\]
Letting $n\rightarrow\infty$, we see from Fatou's lemma
that $\tilde f\in L^{\tilde q}$.  Consequently
$f\in L^{\tilde q}$.

Observe that
\[
  q_n = \frac{r_nN}{N-1} = \frac{N}{N-1}\Big(\frac{q_{n-1}}{{p'_1}} + 1\Big).
\]
Letting $n\rightarrow\infty$, we get
\[
    \tilde q =  \frac{N}{N-1}\Big(\frac{\tilde q}{p'_1} + 1\Big).
\]
Hence $1/p_1 = 1/\tilde q + 1/N$. Therefore, $\tilde q=\tilde p_0$ and
$f\in L^{\tilde p_0}$.

Set $r = \tilde p_0/p'_1+1$. Then (\ref{eq:e12}) turns out to be
\[
 \|f\|_{L^{\tilde p_0}}
 \le  \| f\|_{L^{\tilde p_0}}^{1-1/r}  \Big(Cr  \|\nabla f\|_{L^{p_1}}\Big)^{1/r}.
\]
Finally, we have
\[
 \|f\|_{L^{\tilde p_0}}
 \le  Cr  \|\nabla f\|_{L^{p_1}}
\]
  and this establishes  (\ref{eq:Np1}) since $1/p_1 -1/N= 1/\tilde q   =1/\tilde p_0 $.

Moreover, (\ref{eq:Np1}) implies that $\|f\|_{L^{Np_1/(N-p_1)}} \le C \|f\|_{W_1^{\vec p}}$.
Since $\|f\|_{L^{p_0}}\le  C' \|f\|_{W_1^{\vec p}}$,
by interpolation, we get that
for any $q$ between $p_0$ and $Np_1/(N-p_1)$,
\[
  \|f\|_{L^q} \le C \|f\|_{W_1^{\vec p}} .
\]
This completes the proof.
\end{proof}

The following is an immediate consequence.
\begin{Corollary}\label{Co:C1}
Suppose that $\vec p=(p_0, p_1)$, $\vec q = (q_0, q_1)$,
$1\le p_0,p_1,q_0,q_1<\infty$
and $p_1<N$.
If $q_1=p_1$ and $q_0$ lies between $p_0$ and $Np_1/(N-p_1)$,
then
\[
  W_1^{\vec p}(\bbR^N) \hookrightarrow W_1^{\vec q}(\bbR^N).
\]
\end{Corollary}

\begin{Remark}\upshape
The range for $q$ in the inequality (\ref{eq:e13}) is
the best possible, which can be shown as that for classical Sobolev spaces.
\end{Remark}

For example, take some $f\in W_1^{\vec p}$. Suppose that
(\ref{eq:e13}) is true for some $q$. Replacing $f(\cdot/\lambda)$ for $f$
in (\ref{eq:e13}), we obtain
\[
  \big\| f\big(\frac{\cdot}{\lambda}\big) \big\|_{L^q}
  \lesssim  \big\| f\big(\frac{\cdot}{\lambda}\big) \big\|_{L^{p_0}}
     + \frac{1}{\lambda} \big\| \nabla f\big(\frac{\cdot}{\lambda}\big) \big\|_{L^{p_1}}.
\]
Thus we have
\[
   \lambda^{N/q}\| f\|_{L^q}
  \lesssim \lambda^{N/p_0}  \| f  \|_{L^{p_0}}
     + \frac{1}{\lambda^{1-N/p_1}}\| \nabla f\|_{L^{p_1}}
\]
and therefore
\[
  1 \lesssim \lambda^{N(1/p_0-1/q)} + \lambda^{N(1/p_1-1/N-1/q)}.
\]

First, we assume that $p_0 \le Np_1/(N-p_1)$.
If $q> Np_1/(N-p_1)$, then both $1/p_1-1/N-1/q$
and
$1/p_0-1/q$ are positive. Letting $\lambda\rightarrow 0$, we get a contradiction.

If $q<p_0$, then both $1/p_1-1/N-1/q$
and
$1/p_0-1/q$ are negative. Letting $\lambda\rightarrow \infty $, we also get a contradiction.

Next, we assume that $p_0 > Np_1/(N-p_1)$.
With similar arguments we get a contradiction.
Hence
(\ref{eq:e13}) is true if and only if $q$ is between
$p_0$ and $Np_1/(N-p_1)$.

With the help of Lemma~\ref{Lm:L2}, we
prove the density of
compactly supported
infinitely many differentiable
functions in nonuniform Sobolev spaces
with general  indices.

\begin{Theorem}\label{thm:density}
For any $k\ge 1$ and $\vec p\in [1,\infty)^{k+1}$,
 the space $C_c^{\infty}(\bbR^N)$ is dense in $W_k^{\vec p}(\bbR^N)$.
\end{Theorem}

\begin{proof}
By Lemma~\ref{Lm:density:a},
it suffices to consider the case $1/p_{i_0} > 1/p_{i_0-1} + 1/N$ for
some $i_0$ with $1\le i_0\le k$.

In this case, we have $N\ge 2$.
Let $\vec q =(q_0,\ldots,q_N)$,
where $q_n$ is defined recursively by
\begin{align*}
q_N &= p_N, \\
q_n &= \begin{cases}
         p_n, & \mbox{ if } 1/p_n\ge 1/q_{n+1}- 1/N, \\
         Nq_{n+1}/(N-q_{n+1}),     &\mbox{ otherwise},
\end{cases}
\qquad n=N-1,\ldots, 0.
\end{align*}
By Lemma~\ref{Lm:L2} (ii), we have $f\in W_k^{\vec q}$.
Since  $1/p_n \le 1/q_n \le 1/q_{n-1} + 1/N$ for all $1\le n\le N$,
we have
\begin{equation}\label{eq:e9}
\frac{1}{p_n} \le \frac{1}{q_n} \le \frac{1}{q_m} + \frac{n-m}{N},
  \quad n>m.
\end{equation}

It suffices to show that (\ref{eq:e8}) is also true in this case.

In the sequel we adopt the  notation   introduced in the proof of Lemma~\ref{Lm:density:a}.
Recall that $\partial^{\gamma}f\in L^{q_{|\gamma|}}$.
As in the proof of Lemma~\ref{Lm:density:a}, we may assume that
$\partial^{\gamma}g\in L^{r}$ for any $r>q_{|\gamma|}$.
If $p_{|\alpha|} \ge q_{|\gamma|}$,
then we get (\ref{eq:e8}) as in the proof of Lemma~\ref{Lm:density:a}.
For the case $p_{|\alpha|} < q_{|\gamma|}$,
we see from (\ref{eq:e9}) that
\[
    \frac{1}{p_{|\alpha|}} \le \frac{1}{q_{|\gamma|}} + \frac{|\beta|}{N}.
\]
Hence there is some $r\ge N/|\beta|$ such that
\[
  \frac{1}{p_{|\alpha|}} = \frac{1}{q_{|\gamma|}} + \frac{1}{r}.
\]
Applying H\"older's inequality, we get
\[
  \Big\|\frac{1}{n^{|\beta|}} \partial^{\beta}\psi(\frac{\cdot }{n})
  \partial^{\gamma} g\Big\|_{L^{p_{|\alpha|}}}
\le
 \frac{1}{n^{|\beta|-N/r}}
\cdot \| \partial^{\beta}\psi \|_{L^{r}}
\|\partial^{\gamma} g 1_{\{|x|\ge n\}}
  \|_{L^{q_{|\gamma|}}}
  \rightarrow 0.
\]
This completes the proof.
\end{proof}

Recall that for an integer $n\ge 0$,
the space $C_b^n(\bbR^N)$
consists of all functions $f$ such that
$f$ is $n$ times continuously differentiable and
for any multi-index $\alpha$ with $|\alpha|\le n$,
$D^{\alpha}f\in L^{\infty}$.

For an integer $n\ge 0$ and a positive number $\nu\in (0,1)$,
the space $C^{n,\nu}(\bbR^N)$ consists of all functions $f$ in $C_b^n$ such that
for any multi-index $\alpha$ with $|\alpha|=n$,
\[
  | D^{\alpha} f (x) - D^{\alpha}f(y)|
  \le C_{n,\nu} |x-y|^{\nu},\quad \forall x,y\in\bbR^N.
\]

Below is an embedding theorem for higher-order nonuniform Sobolev spaces.

\begin{Theorem}\label{thm:embedding}
Let $k$ be a positive integer and $\vec p = (p_0, \ldots, p_k)
\in [1,\infty)^{k+1}$.

\begin{enumerate}
\item If $kp_k<N$,  then for any $q$ between $p_0$ and $Np_k/(N-kp_k)$, we have
    \[
      W_k^{\vec p} (\mathbb R^N)  \hookrightarrow L^q (\mathbb R^N) .
    \]

\item If $kp_k=N$, then for any $q$ with $p_0\le q<\infty$, we have
    \[
      W_k^{\vec p} (\mathbb R^N)  \hookrightarrow L^q(\mathbb R^N) .
    \]

\item If $kp_k>N$,  then $
      W_k^{\vec p} (\mathbb R^N)  \hookrightarrow L^{\infty} (\mathbb R^N)$.
      More precisely,
if  $kp_k>N$ and $N/p_k\not\in\bbZ$,
then there is some integer $k_0$
such that $(k_0-1)p_k <N<k_0 p_k$
and
    \[
      W_k^{\vec p} (\mathbb R^N)  \hookrightarrow
      C_b^{k-k_0,k_0-N/p_k} (\mathbb R^N) .
    \]
 If  $N/p_k \in\bbZ$ and $k\ge k_0 :=N/p_k+1$,
then for any $0<\lambda<1$,
 \[
      W_k^{\vec p} (\mathbb R^N)  \hookrightarrow
      C_b^{k-k_0,\lambda} (\mathbb R^N) .
    \]

\end{enumerate}
\end{Theorem}

\begin{proof}
(i)\,\, First, we consider the case $kp_k<N$.
We see from
Lemma~\ref{Lm:L2} that the conclusion is true for $k=1$.

For the case $k\ge 2$, applying (\ref{eq:Np1}) recursively, we get
\begin{align}
 \|f\|_{L^{Np_k/(N-kp_k)}}
 &\le C_1 \sum_{|\alpha|=1} \|\partial^{\alpha} f\|_{L^{Np_k/(N-(k-1)p_k)}}
    \nonumber \\
 &\le C_2 \sum_{|\alpha|=2} \|\partial^{\alpha} f\|_{L^{Np_k/(N-(k-2)p_k)}}
   \nonumber \\
 &\le \ldots \nonumber \\
 &\le C_k \sum_{|\alpha|=k } \|\partial^{\alpha} f\|_{L^{p_k}}. \label{eq:e14}
\end{align}
By interpolation, we get that for any  $q$ between $p_0$ and $Np_k/(N-kp_k)$,
$  W_k^{\vec p}
  \hookrightarrow L^q$.

\medskip\noindent
(ii)\,\, Next, we consider the case $kp_k=N$.

First, we assume that $k=1$ and $N>1$.
Take some $f\in W_1^{\vec p}$, where $\vec p = (p_0, N)$.
Set $r = p_0(N-1)/N+1$ and $h = |f|^{r-1} f$.
We have $\nabla h = r |f|^{r-1} \nabla f$.
Since $(r-1)/p_0 + 1/N = 1$, we see from
H\"older's inequality that
\[
 \|\nabla h\|_{L^1}
 \le  r\|f\|_{L^{p_0}}^{r-1} \|\nabla f\|_{L^N}<\infty.
\]
Hence $\nabla h\in L^1$.
By Lemma~\ref{Lm:L2}, $h\in L^{N/(N-1)}$.
Therefore, $f\in L^{rN/(N-1)} = L^{p_0+N/(N-1)}$.

Replacing $p_0$ by $p_0+N/(N-1)$  in the preceding argument,
we obtain  that $f$ lies in $ L^{p_0+2N/(N-1)}$.
Repeating this procedure yields that
$f\in L^{p_0+nN/(N-1)}$ for any $n\ge 1$.
Hence
$f\in L^q$ for any $q$ satisfying
$p_0\le q<\infty$.

For the case $k=N=1$, we see from (\ref{eq:e5}) that for smooth functions $f$,
\[
  \|f\|_{L^{\infty}} \lesssim \|f'\|_{L^1}.
\]
Applying the density of $C_c(\bbR)$ in $W_1^{\vec p}(\bbR)$, we get
that the above inequality is valid for all $f\in W_1^{\vec p}$.
Hence
$f\in L^q$ for any $q$ satisfying
$p_0\le q<\infty$.

Next we assume that $kp_k=N$ for some $k\ge 2$.
Take some $f\in W_k^{\vec p}$, where $\vec p = (p_0,\ldots, p_k)$.
For any multi-index $\alpha$ with $|\alpha|=1$, we have
$\partial^{\alpha} f \in W_{k-1}^{(p_1,\ldots,p_k)}$.
Note that  $(k-1)p_k < N$. We see from (i) that
$\partial^{\alpha} f\in L^{Np_k/(N-(k-1)p_k)} = L^N$.
Hence $f\in W_1^{(p_0, N)}$. Now we see from   arguments in
  the case $k=1$ that $f\in L^q$ for any $q$ satisfying
$p_0\le q<\infty$.

\medskip\noindent
(iii)\,\, Finally, we consider the case $kp_k>N$.

First, we show that $W_k^{\vec p}\subset C_b^0$.
We prove it by induction on $k$.
For $k=1$, the arguments in \cite[pages 80-82]{Demengel2012}
work well for the nonuniform case with minor changes.
Specifically, applying the $L^{p_1}$ norm for derivatives and
the $L^{p_0}$ norm for the function itself,
the same arguments yield that $W_1^{\vec p}\hookrightarrow C_b^0$. That is, functions in $W_1^{\vec p}$ are continuous
and bounded.

Now we assume that the conclusion is true
for the cases $1$, $\ldots$, $k-1$.
Consider the case $k$.
Fix some $f\in W_k^{\vec p}$.
There are two cases:

\medskip\noindent
(a)\,\,  $(k-1)p_k\ge N$.

If $(k-1)p_k > N$, then we see from the inductive assumption
that for any multi-index $\alpha$ with $|\alpha|=1$,
\[
 \|\partial^{\alpha} f \|_{L^{\infty}} \lesssim
 \|\partial^{\alpha} f\|_{W_{k-1}^{(p_1,\ldots,p_k)}}
 \le \| f\|_{W_k^{\vec p}}.
\]
If $(k-1)p_k=N$, then we see from (ii) that
for any multi-index $\alpha$ with $|\alpha|=1$,
$\partial^{\alpha} f\in L^q$ for all $q$ with $p_0\le q<\infty$.

Hence for $(k-1)p_k\ge N$, there is some $q>\max\{N,p_0\}$ such that
\[
  \|\nabla f\|_{L^q} \lesssim \| f\|_{W_k^{\vec p}}.
\]
Consequently, $f\in W_1^{(p_0,q)}$.
Applying the inductive assumption  for the case $k=1$, we get that $f$ is continuous and
\[
  \| f\|_{L^{\infty}} \lesssim \|f\|_{W_1^{(p_0,q)}} \lesssim \| f\|_{W_k^{\vec p}}.
\]

\medskip\noindent
(b)\,\,  $(k-1)p_k< N$.

In this case, we see from (i) that $f\in W_1^{(p_0,q_1)}$
with $q_1 = Np_k/(N-(k-1)p_k)$. Since $kp_k>N$, we have $q_1>N$.
Applying the inductive assumption  for the case $k=1$ again, we get that $f$ is continuous and
\[
  \| f\|_{L^{\infty}} \lesssim \|f\|_{W_1^{(p_0,q_1)}}  \lesssim \| f\|_{W_k^{\vec p}}.
\]
By induction, the embedding
$  W_k^{\vec p}
\hookrightarrow  C_b^0$
 is valid for all $k\ge 1$.

Next we prove the H\"older continuity.

First, we assume that $N/p_k$ is not an integer.
When $k=1$, we have $p_1>N$.
Take some $f\in W_1^{\vec p}$.
Let $\varphi$ and $g_{\lambda}$  be defined as in Lemma~\ref{Lm:density}.
We have
\[
  g_{\lambda}(x+y) - g_{\lambda}(x)
     = \int_0^1 \nabla g_{\lambda}(x+ty)\cdot y \rmd t
     = \int_0^1\int_{\bbR^N}
     \nabla f(x+ty-z)\cdot y  \varphi(z) \rmd z\rmd t
     ,
     \quad \forall x,y\in\bbR^N.
\]
It follows from Minkowski's inequality that
\[
 \| g_{\lambda}(\cdot+y) - g_{\lambda}\|_{L^{p_1}}
 \le |y|\cdot  \|\nabla f\|_{L^{p_1}}.
\]
Since $f$ is continuous and $\lim_{\lambda\rightarrow 0}g_{\lambda}(x) = f(x)$ for all $x\in\bbR^N$,
letting $\lambda\rightarrow 0$ in the above inequality,
we see from Fatou's lemma that
\[
 \| f(\cdot+y) - f\|_{L^{p_1}}
 \le |y|\cdot  \|\nabla f\|_{L^{p_1}}.
\]
On the other hand, it is easy to see that
\[
 \| \nabla(f(\cdot+y) - f)\|_{L^{p_1}}
 \le 2 \|\nabla f\|_{L^{p_1}}.
\]
Applying the fact for  classical Sobolev spaces
that if $u, |\nabla u| \in L^{p_1}$ and $p_1>N$,
then
\begin{equation}\label{eq:s:e4}
  \|u\|_{L^{\infty}} \lesssim \|u\|_{L^{p_1}}^{1-N/p_1}
      \| \nabla u  \|_{L^{p_1}}^{N/p_1},
\end{equation}
we get
\[
  \| f(\cdot+y) - f\|_{L^{\infty}}
 \le |y|^{1-N/p_1}\cdot  \|\nabla f\|_{L^{p_1}}.
\]
Hence $f\in C_b^{0,1-N/p_1}$.

When $k\ge 2$,  there is some positive integer
$k_0\le k$ such that
$(k_0-1)p_k < N<k_0 p_k$.
It follows that for any multi-index $\alpha$
with $|\alpha|=k-k_0$, we have
$D^{\alpha} f \in W_{k_0}^{(p_{k-k_0},\ldots,p_k)}$.
Now we see from (i) that
for any multi-index $\beta$
with $|\beta|=1$,
\[
  D^{\alpha+\beta} f \in L^{Np_k/(N-(k_0-1)p_k)}.
\]
Since $Np_k/(N-(k_0-1)p_k)>N$,
applying the conclusion for $k=1$, we get
$D^{\alpha} f\in
C_b^{0,1-N/(Np_k/(N-(k_0-1)p_k))} =
C_b^{0,k_0-N/p_k}$.

Moreover,
by the embedding inequality we have proved,
 $D^{\beta}f\in L^q$
whenever $1\le |\beta|\le k-k_0$ and $q$ is large enough.
In particular, $D^{\beta}f\in L^q$ for some $q>N$.
Applying the conclusion for $k=1$ again, we get
$D^{\alpha} f\in C_b^0$ when $|\alpha|\le k-k_0-1$.
 Hence $f\in C_b^{k-k_0,k_0-N/p_k}$.

It remains to consider the case $N/p_k\in \bbZ$.
When $(k_0-1)=N<k_0p_k$,
for any multi-index $\alpha$
with $|\alpha|=k-k_0$, we have
$D^{\alpha} f \in W_{k_0}^{(p_{k-k_0},\ldots,p_k)}$.
Now we see from (ii) that for $q$ large enough
and $|\beta|=1$,
\[
  D^{\alpha+\beta} f \in L^q.
\]
The arguments for the case $k=1$ show that
for $q$ large enough, $D^{\alpha}f\in C_b^{0,1-N/q}$.
Consequently, $D^{\alpha}f\in C_b^{0,\lambda}$
whenever $|\alpha|=k-k_0$ and $0<\lambda<1$.

On the other hand,
the same arguments as those
for the case $N/p_k\not\in\bbZ$ show that
$D^{\alpha} f\in C_b^0$ when $|\alpha|\le k-k_0-1$.
Hence $f\in C_b^{k-k_0,\lambda}$ for any $0<\lambda<1$.
\end{proof}

Theorem~\ref{thm:embedding} has some interesting consequences.
In fact, the following corollary can be
 proved with similar arguments as in   the proof
 of Theorem~\ref{thm:embedding} (iii),
 for which we leave the details to interested readers.

\begin{Corollary}\label{Co:embedding}
Let $k$ be a positive integer and $\vec p = (p_0, \ldots, p_k)
\in [1,\infty)^{k+1}$.
If $p_k>N$, then
\[
  W_k^{\vec p}  \hookrightarrow   C_b^{k-1, 1-N/p_k}.
\]
\end{Corollary}

On the other hand, the following result shows that the Sobolev space
$W_k^{\vec q}$ with $1\le q_k<N/k$ and $1/q_i = 1/q_k - (k-i)/N$
is the largest one among all
Sobolev spaces
$W_k^{\vec p}$ with $p_k=q_k$.

\begin{Corollary}\label{Co:C2}
Suppose that $\vec p = (p_0,\ldots, p_k)\in [1,\infty)^{k+1}$
with $1\le p_k<N/k$. Let
$\vec q = (q_0,\ldots,q_k)$
be such that $q_k=p_k$ and $q_i = Np_k/(N-(k-i)p_k)$ for $0\le i\le k-1$.
Then we have
\[
  W_k^{\vec p}  \hookrightarrow W_k^{\vec q} .
\]

Moreover, set $\vec q^{(i)} = (q_0,\ldots, q_i)$,
$1\le i\le k$.
We have
\[
  W_k^{\vec q^{(k)}}
  \hookrightarrow W_{k-1}^{\vec q^{(k-1)}}
  \hookrightarrow \cdots
  \hookrightarrow W_1^{\vec q^{(1)}}
  \hookrightarrow L^{q_0}.
\]
\end{Corollary}

\begin{proof}
We see from (\ref{eq:e14}) that for any $f\in W_k^{\vec p}$,
\[
  \|f\|_{W_k^{\vec q}} \le C  \sum_{|\alpha|=k } \|\partial^{\alpha} f\|_{L^{p_k}}
  \le C \|f\|_{W_k^{\vec p}}.
\]
Hence $  W_k^{\vec p}  \hookrightarrow W_k^{\vec q}$.
The second conclusion is obvious. This completes the proof.
\end{proof}

\section{Nonuniform Fractional Sobolev Spaces}

Recall that
nonuniform fractional Sobolev spaces
are introduced in Definition~\ref{def:D1}.
We point out that for any $0<s<1$ and $\vec p = (p_0,p_1)\in [1,\infty)^2$,
 $W_s^{\vec p}(\bbR^N)$ is a Banach space,
 a fact that can be proved with almost the same arguments as that  used in the proof of
\cite[Proposition 4.24]{Demengel2012}.

\subsection{Embedding Theorem for Nonuniform Fractional Sobolev Spaces}

The main result in this subsection is the following embedding theorem.

\begin{Theorem}\label{thm:embedding:fractional}
Suppose that $\vec p = (p_0, p_1)$ with $1\le p_0<\infty$, $1<p_1<\infty$
and  $0<s<1$.
When $ p_0\le p_1 $, we have
\begin{enumerate}
\item If
 $sp_1<N$, then $W_s^{\vec p}(\bbR^N)
   \hookrightarrow L^q(\bbR^N)$ for all $p_0\le q \le Np_1/(N-sp_1)$.

\item  If   $sp_1=N$, then $W_s^{\vec p}(\bbR^N)
   \hookrightarrow L^q(\bbR^N)$ for all $p_0\le q <\infty$.

\item   If  $sp_1>N$, then $W_s^{\vec p}(\bbR^N)
   \hookrightarrow C_b^{0,s-N/p_1}(\bbR^N)$.
\end{enumerate}

For the case $p_0>p_1$, if $sp_1<N$ and  $ p_0< Np_1/(N-sp_1)$, we also have
$W_s^{\vec p}(\bbR^N)
   \hookrightarrow L^q(\bbR^N)$ for all  $p_0 \le q\le Np_1/(N-sp_1)$.
\end{Theorem}

To prove Theorem~\ref{thm:embedding:fractional}, we need some preliminary results.

We see from \cite[Lemma 4.33]{Demengel2012} that
for any $f\in L^p$,
\[
  \iint_{\bbR^N\times\bbR^N} \frac{|f(x)-f(y)|^p}{|x-y|^{N+ps}} \rmd x\,\rmd y<\infty
\]
is equivalent to
\[
\iint_{\bbR^N\times\bbR} \frac{|f(x)-f(x+ae_j)|^p}{|a|^{1+ps}} \rmd x\,\rmd a<\infty,
\qquad \forall\,\, 1\le  j\le N.
\]
Checking the arguments in the proof
of \cite[Lemma 4.33]{Demengel2012},
we find that
the hypothesis $f\in L^p$ is not used.
Moreover, with the same arguments
we get the following
result.

\begin{Proposition}\label{prop:p1}
Let $0<s<1$ and $1\le p<\infty$.\
Then
there exist positive constants $C_1$ and $C_2$
such that for any
 measurable function $f$ which is finite almost everywhere,
\[
 C_1 [f]_{W_s^p}^p
\le
  \sum_{j=1}^N\iint_{\bbR^N\times\bbR} \frac{|f(x)-f(x+ae_j)|^p}{|a|^{1+ps}} \rmd x\,\rmd a
  \le C_2 [f]_{W_s^p}^{p},\qquad \forall\,\, 1\le  j\le N.
\]
\end{Proposition}

The following construction is used in the proof of the embedding theorem.
Unlike the classical case, here we need that $\varphi $ equals $1$
in a neighborhood of $0$
to ensure that its derivatives vanish near $0$.

\begin{Lemma} \label{Lm:Ls1}
Suppose that $1\le  p_0, p_1<\infty$, $0<s\le 1$,
$f\in W_s^{\vec p}(\bbR^N)$, $\varphi\in C_c^{\infty}(\bbR)$
with $\varphi(t)=1$ for $|t|<1$ and $\varphi(t)=0$ for $|t|>A$,
where $A$ is a constant.
Let
 $\alpha = 1 - 1/p_1 - s$ for $0<s<1$ and $\alpha = -1/p_1+\eta$ for $s=1$,
 where $0<\eta<1$.
Set
\begin{align}\label{eq:g}
  g(t,x) &= \frac{\varphi(t)}{t^N}\int_{[0,t]^N} f(x+y)\, \rmd y,\quad t>0, \\
  g(0,x) &= \lim_{t\rightarrow 0} g(t,x).
\end{align}
We have
\begin{align}
\left\| t^{\alpha} \nabla_x g(t,x)  \right\|_{L^{p_1}(\bbR^{N+1})}
&\lesssim  \|f\|_{W_s^{\vec p}}.  \label{eq:s:11}
\end{align}
Moreover, there are two functions $h_0$ and $h_1$ such that $\partial_t g(t,x)
= h_0(t,x)+h_1(t,x)$ and
\begin{align}
\left\| t^{\alpha}  h_0(t,x)  \right\|_{L^{p_1}(\bbR^{N+1})}
&\lesssim  \|f\|_{W_s^{\vec p}},   \label{eq:s:11b} \\
\left\|  h_1 (t,\cdot ) \right\|_{L^{p_0}(\bbR^N)}
&\lesssim  \|f\|_{L^{p_0}} 1_{[1,A]}(t).  \label{eq:s:11c}
\end{align}
Furthermore, if $p_0\le p_1$, we have
\begin{align}
\left\| t^{\alpha} \nabla g (t,x) \right\|_{L^{p_1}(\bbR^{N+1})}
&\lesssim  \|f\|_{W_s^{\vec p}}.  \label{eq:s:11d}
\end{align}
\end{Lemma}

\begin{proof}
First, we estimate
$\| \partial_{x_i} g(t,x) \|_{L^{p_1}(\bbR^{N+1})}$,
$1\le i\le N$.
Denote
\begin{align*}
\hat x_i &= x- x_ie_i  \mbox{\quad and\quad}
\rmd\hat x_i  = \prod_{j\ne i} \rmd x_j.
\end{align*}
Observe
\begin{align*}
g(t,x) &= \frac{\varphi(t)}{t^N}\int_{x_i}^{x_i+t}
    \int_{[0,t]^{N-1}} f(\hat x_i + \hat y_i + y_ie_i)
    \,  \rmd \hat y_i \,  \rmd y_i.
\end{align*}
We get
\begin{align}
 \partial_{x_i} g(t,x)
&= \frac{\varphi(t)}{t^N}
    \int_{[0,t]^{N-1}} \Big(f(\hat x_i + \hat y_i + (x_i+t)e_i) - f(\hat x_i + \hat y_i +  x_i e_i)
    \Big)
      \rmd \hat y_i  \nonumber \\
&= \frac{\varphi(t)}{t^N}
    \int_{[0,t]^{N-1}} \Big(f(x+ te_i+\hat y_i   )
     - f(  x + \hat y_i  )
    \Big)
      \rmd \hat y_i \nonumber \\
&= \frac{\varphi(t)}{t }
    \int_{[0,1]^{N-1}} \Big(f(x+ te_i+t\hat y_i   )
     - f(  x + t\hat y_i  )
    \Big)
      \rmd \hat y_i.  \nonumber
\end{align}
Hence for $0<s<1$,
\begin{align*}
\| t^{\alpha}\partial_{x_i} g(t,x) \|_{L^{p_1}(\bbR^{N+1})}^{p_1}
\!&\le
  \!\| \varphi(t)\|_{L^{\infty}}^{p_1}
   \int_0^{\infty}\!\!\!
   \int_{\bbR^N}\!\!
   \int_{[0,1]^{N-1}}\!\!\!
     \frac{|f(x+ te_i+t\hat y_i   )
     - f(  x + t\hat y_i  )|^{p_1}}
     {t^{1+sp_1}} \rmd \hat y_i \,  \rmd x\,    \rmd t \\
&\lesssim [f]_{W_s^{p_1}}^{p_1}.
\end{align*}

For the case $s=1$, we have $\alpha = -1/p_1+\eta$.
We see from (\ref{eq:g}) that
\[
  \nabla_x g(t,x)
  = \frac{\varphi(t)}{t^N} \int_{[0,t]^N} \nabla f(x+y)\rmd y.
\]
Thus we obtain
\begin{align*}
\|t^{\alpha}\nabla_x g(t,x)\|_{L^{p_1}(\bbR^{N+1})}
\le  \|t^{\alpha} \varphi(t)\|_{L^{p_1}} \|\nabla f\|_{L^{p_1}}<\infty
\end{align*}
proving (\ref{eq:s:11}).

Next we deal with $\partial_t g(t,x)$.
A simple computation shows that
\begin{align*}
 \partial_t g(t,x)
&= \varphi(t)\bigg(
    \frac{-N}{t^{N+1}}\int_{[0,t]^N} f(x+y)\rmd y
     + \frac{1}{t^N} \sum_{i=1}^N
    \int_{[0,t]^{N-1}} f(x+\hat y_i + te_i) \rmd \hat y_i
   \bigg) \\
 &\qquad + \frac{\varphi'(t)}{t^N} \int_{[0,t]^N} f(x+y)\rmd y\\
&=  \frac{\varphi(t)}{t^{N+1}}
   \sum_{i=1}^N \int_{[0,t]^N}
      \Big(f(x+\hat y_i + te_i)
        - f(x+y)\Big)\rmd y
       +  \varphi'(t)
        \int_{[0,1]^N} f(x+ty)\rmd y\\
&=  \frac{\varphi(t)}{t}
   \sum_{i=1}^N \int_{[0,1]^N}
      \Big(f(x+t\hat y_i + te_i)
        - f(x+ty)\Big)\rmd y
       +  \varphi'(t)
        \int_{[0,1]^N} f(x+ty)\rmd y\\
&:= h_0(t,x)+ h_1(t,x).
\end{align*}

Next, we prove (\ref{eq:s:11b}).
For $0<s<1$,  we have
\begin{align*}
\|t^{\alpha}h_0(t,x)\|_{L^{p_1}(\bbR^{N+1})}^{p_1}
&\lesssim \|\varphi\|_{L^{\infty}}^{p_1}
   \sum_{i=1}^N
     \int_0^{\infty}\!
       \int_{[0,1]^N}\!
      \int_{\bbR^N}
      \frac{|f(x+t\hat y_i + te_i)
        - f(x+ty)|^{p_1}}{t^{1+sp_1}}
      \rmd x  \,\rmd y \, \rmd t.
\end{align*}
By the change of variables
$x\rightarrow x+ty$, we write
\begin{align}
\|t^{\alpha}h_0(t,x)\|_{L^{p_1}(\bbR^{N+1})}^{p_1}
&\lesssim \|\varphi\|_{L^{\infty}}^{p_1}
   \sum_{i=1}^N
     \int_0^{\infty}
       \int_{[0,1]^N}
      \int_{\bbR^N}
      \frac{|f(x+ t(1-y_i)e_i)
        - f(x)|^{p_1}}{t^{1+sp_1}}
      \rmd x  \,\rmd y \, \rmd t \nonumber \\
&=  \|\varphi\|_{L^{\infty}}^{p_1}
   \sum_{i=1}^N
     \int_0^{\infty}
       \int_0^1(1-y_i)^{sp_1}
      \int_{\bbR^N}
      \frac{|f(x+ t e_i)
        - f(x)|^{p_1}}{t^{1+sp_1}}
      \rmd x  \,\rmd y_i \, \rmd t  \nonumber \\
&\lesssim [f]_{W_s^{p_1}}^{p_1},   \label{eq:s:14}
\end{align}
having used the change of variables $t\to (1-y_i)t$.
And for the case $s=1$,
we have
\begin{align*}
\|t^{\alpha}h_0(t,x)\|_{L_x^{p_1}}
&\le
|t^{\alpha-1}\varphi(t)|
   \sum_{i=1}^N \int_{[0,1]^N}
      \Big\|f(x+t\hat y_i + te_i)
        - f(x+ty)\Big\|_{L_x^{p_1}}\rmd y \\
&=
|t^{\alpha-1}\varphi(t)|
   \sum_{i=1}^N \int_{[0,1]^N}
      \Big\| \int_{ty_i}^{t} \partial_{x_i}f(x+t\hat y_i + \tau e_i)
      \rmd \tau\Big\|_{L_x^{p_1}}\rmd y \\
&\lesssim |t^{\alpha} \varphi(t)|      \cdot  \|\nabla f\|_{L^{p_1}}.
\end{align*}
Hence
\begin{equation}\label{eq:t1}
\|t^{\alpha}h_0(t,x)\|_{L^{p_1}(\bbR^{N+1})}\lesssim \|\nabla f\|_{L^{p_1}}.
\end{equation}

To prove (\ref{eq:s:11c}), note that   $\varphi'(t)= 0$
for $0<t<1$ or $t>A$. For $1\le t\le A$, we see
from Minkowski's inequality that
\begin{align}
\|h_1(t,\cdot )\|_{L^{p_0}(\bbR^N)}
&\le  |\varphi'(t)|  \int_{[0,1]^N} \|f(\cdot +ty)\|_{L^{p_0}}\rmd y
\le  |\varphi'(t)| \cdot  \|f\|_{L^{p_0}}.
\end{align}
Hence (\ref{eq:s:11c}) is valid.

Finally, we consider the case $p_0\le p_1$.
To prove (\ref{eq:s:11d}), it suffices to show that
$
\|t^{\alpha}h_1(t,x)\|_{L^{p_1}(\bbR^{N+1})} \lesssim \|f\|_{L^{p_0}}$,
thanks to (\ref{eq:s:11}) and (\ref{eq:s:11b}).

Recall that   $\varphi'(t)= 0$
for $0<t<1$. For $t\ge 1$, we have
\begin{align}
|h_1(t,x)|
&\le  |\varphi'(t)|  \Big(\int_{[0,1]^N} |f(x+ty)|^{p_0}
  \rmd y\Big)^{1/p_0} \nonumber \\
&= |\varphi'(t)|
\Big(\frac{1}{t^N}\int_{[0,t]^N} |f(x+y)|^{p_0}
  \rmd y\Big)^{1/p_0} \nonumber \\
&\le |\varphi'(t)|\cdot
\|f\|_{L^{p_0}}. \nonumber 
\end{align}
Hence
\begin{align}
\|t^{\alpha}h_1(t,x)\|_{L^{p_1}(\bbR^{N+1})}^{p_1}
&\le\int_{\bbR^N} \int_1^{\infty} |t^{\alpha}\varphi'(t)|^{p_1}
   \|f\|_{L^{p_0}}^{p_1-p_0}
   \Big(\int_{[0,1]^N} |f(x+ty)|
  \rmd y\Big)^{p_0}  \,\rmd t\, \rmd x\nonumber  \\
&\le\int_{\bbR^N} \int_1^{\infty} |t^{\alpha}\varphi'(t)|^{p_1}
   \|f\|_{L^{p_0}}^{p_1-p_0}
    \int_{[0,1]^N} |f(x+ty)|^{p_0}
\,  \rmd y  \,\rmd t\, \rmd x\nonumber \\
&\le \|f\|_{L^{p_0}}^{p_1} . \nonumber 
\end{align}
This completes the proof.
\end{proof}

For the case  $ 1\le  p_0\le   p_1 <\infty$,
the nonuniform Sobolev space $W_s^{(p_0,p_1)}$
is a subspace of the classical one $W_s^{p_1}$.
Moreover, the inclusion is in fact an embedding
and is also valid for $s=1$.

\begin{Lemma}\label{Lm:L3a}
Suppose that $0<s\le 1$ and $1\le p_0\le  p_1<\infty$. We have
\[
  W_s^{\vec p} \hookrightarrow W_s^{p_1} .
\]
Moreover, the inclusion is not true if $p_0>p_1$.
\end{Lemma}

\begin{proof}
Let $\varphi$ and $\psi$  be functions in
$C_c^{\infty}(\bbR^N)$
 and $C_c^{\infty}(\bbR)$, respectively, with values
between $0$ and $1$ and equal to $1$ in neighborhoods of $0$.
Set $\Phi(t,x) = \varphi(x)\psi(t)$.
Take some $A>0$ such that $\mathrm{supp}\, \Phi \subset  [-A,A]^{N+1}$.

Fix some $f\in W_s^{\vec p}$.
Let $\alpha $ and $g$ be defined as in Lemma~\ref{Lm:Ls1},
$\delta_0$ be the Dirac measure at $0$,
and
\[
  E(t,x) = \begin{cases}
     \dfrac{c_N}{(t^2+|x|^2)^{(N-1)/2}}, & N\ge 2,\\
     c_1 \log (t^2+|x|^2), &N=1,
  \end{cases}
\]
 be the fundamental solution of the Laplacian satisfying
$
  \Delta E = \delta_0$.

Recall that for a tempered distribution $\Lambda$ and a function $h\in C^{\infty}$, for  which the function itself and all of its derivatives have at most polynomial
growth at infinity,
the product
$h  \Lambda $ is the  tempered  distribution \cite[Definition 2.3.15]{Grafakos2014} defined by
\begin{equation}\label{eq:s:e1}
\langle h  \Lambda, u \rangle
 =\langle  \Lambda, h  u\rangle, \qquad \forall u\in\mathscr S.
\end{equation}
Since $\Phi(t,x)$ equals $1$ in a neighbourhood of $0$, we have
\[
     \Phi \Delta E = \Phi \delta_0 = \delta_0.
\]
Hence
\begin{align}
g & = \delta_0 * g \nonumber \\
  & = \Delta(\Phi E) *g
      -2 (\nabla \Phi \cdot \nabla E) *g
        - ((\Delta\Phi)E) *g
        \nonumber \\
  &=   ( (\nabla \Phi) E) * \nabla g + (\Phi \nabla  E )* \nabla g
      -2 (\nabla \Phi \cdot \nabla E) *g
        - ((\Delta\Phi) E) *g
          ,\label{eq:s:e17}
\end{align}
where we use the notation
\[
   u * v = \sum_{i=1}^{N+1} u_i * v_i
\]
to denote the convolution of two $\mathbb C^{N+1}$-valued
functions.

Since $g(0,x) = f(x)$ a.e., it suffices to show  that $\|g(0,\cdot)\|_{L^{p_1}}\lesssim \|f\|_{W_s^{\vec p}}$.
We prove it in the following steps:

\medskip\noindent
(S1)\,
First, we estimate $\|((\nabla \Phi) E) * \nabla g (0,\cdot) \|_{L^{p_1}}$.
For $N\ge 2$, we have
\[
  ((\nabla \Phi) E )* \nabla g  (0,x)
  = \int_0^A  \int_{\bbR^N} \frac{ (\nabla \Phi)(-t,y)\cdot  \nabla g(t,x-y)}
         {(t^2 + |y|^2)^{(N-1)/2}} \rmd y  \,\rmd t.
\]
Applying Minkovski's inequality, we get
\begin{align*}
\|((\nabla \Phi) E )* \nabla g (0,\cdot) \|_{L^{p_1}}
&\le \| \nabla \Phi\|_{L^{\infty}}
  \int_0^A \| \nabla g(t,\cdot)\|_{L^{p_1}}   \int_{[-A,A]^N} \frac{\rmd y}
         {(t  + |y| )^{ N-1 }} \,\rmd t \\
&\lesssim
  \int_0^A \| \nabla g(t,\cdot)\|_{L^{p_1}}   \int_{[0,A]^2} \frac{ \rmd y_1 \rmd y_2}
         { t + y_1 + y_2 }   \,\rmd t \\
&\approx
  \int_0^A \| \nabla g(t,\cdot)\|_{L^{p_1}}    \int_0^A \log \frac{t+y_1+A}{t+y_1}
   \rmd y_1 \,\rmd t
            \\
&\lesssim
  \int_0^A \| \nabla g(t,\cdot)\|_{L^{p_1}}   \int_0^A  \frac{1}{(t+y_1)^{\varepsilon}}
     \rmd y_1  \,\rmd t
           \mbox{\qquad  $(0<\varepsilon<1)$} \\
&\lesssim
  \int_0^A \| \nabla g(t,\cdot)\|_{L^{p_1}}  \rmd t \\
&\lesssim
   \| t^{-\alpha  } \cdot 1_{[0,A]}(t) \|_{L^{p'_1}}
    \| t^{\alpha}\nabla g(t,x)\|_{L^{p_1}(\bbR^{N+1})}   \\
&\approx  \| t^{\alpha}\nabla g(t,x)\|_{L^{p_1}(\bbR^{N+1})}\\
&\lesssim \|f\|_{W_s^{\vec p}},
\end{align*}
where we applied (\ref{eq:s:11d}) in the last step.
When $N\ge 3$, the first inequality is obtained by successively
integrating the variables $y_N,y_{N-1},\dots, y_3$ over $[0,\infty)$;
the effect of this is the reduction of the exponent by $1$ in each integration, so after $N-2$ integrations,  the exponent
becomes  $(N-1)-(N-2)=1$.

For the case  $N=1$, we have $|E(t,y)| = |c_N \log (t^2 + y^2)|\lesssim
 (|t|+|y|)^{-\varepsilon}$ for any $\varepsilon>0$.
Hence
\begin{align*}
\|((\nabla \Phi) E) * \nabla g (0,\cdot) \|_{L^{p_1}}
&\lesssim
  \int_0^A \| \nabla g(t,\cdot)\|_{L^{p_1}}   \int_0^A  \frac{1}{(t+y_1)^{\varepsilon}}
     \rmd y_1   \,\rmd t
           \mbox{\qquad  $(0<\varepsilon<1)$} \\
&\lesssim \|f\|_{W_s^{\vec p}}.
\end{align*}

\medskip\noindent
(S2)\,
Next,  we estimate $\|  (\Phi  \nabla E ) * \nabla g (0,\cdot) \|_{L^{p_1}}$.
Observe that for $(t,y)\in [-A,A]^{N+1}$,
\[
  |\nabla E(t,y)| \lesssim \frac{1}{(|t|+|y|)^N}.
\]
Similar arguments as the previous case show that
\begin{align*}
\|(\Phi  \nabla E )* \nabla g (0,\cdot) \|_{L^{p_1}}
&\lesssim
  \int_0^A \| \nabla g(t,\cdot)\|_{L^{p_1}}   \int_{[-A,A]^N} \frac{\rmd y}
         {(t + |y|)^N }  \,\rmd t \\
&\approx
  \int_0^A \| \nabla g(t,\cdot)\|_{L^{p_1}}   \int_0^A \frac{ \rmd y_1 }
         { t + y_1  }  \,\rmd t\\
&\approx
  \int_0^A \| \nabla g(t,\cdot)\|_{L^{p_1}}   \log \frac{t+A}{t}   \rmd t
            \\
&\lesssim
  \int_0^A \| \nabla g(t,\cdot)\|_{L^{p_1}}    t^{-\varepsilon} \rmd t
          \\
&\lesssim
   \| t^{-(\alpha+\varepsilon) } \cdot 1_{[0,A]} (t)\|_{L^{p'_1}}
    \| t^{\alpha}\nabla g(t,x)\|_{L^{p_1}(\bbR^{N+1})}   \\
&\lesssim \|f\|_{W_s^{\vec p}},
\end{align*}
where $0<\varepsilon<s$ for $s<1$ and $0<\varepsilon<1-\eta$ for $s=1$.
Again, we apply (\ref{eq:s:11d}) in the last step.

\medskip\noindent
(S3)\,
We deal with the last two terms.
Observe that both  $\nabla \Phi$ and $\Delta \Phi$ vanish in a neighborhood
of $0$. Hence both $(\nabla \Phi \cdot \nabla E)$ and $(\Delta\Phi) E$
are compactly supported bounded functions.
Hence
\begin{align*}
F(x):=|2(\nabla \Phi \cdot \nabla E) *g(0,x)|
+|((\Delta \Phi)E) * g(0,x)|
&\lesssim \int_0^A \int_{[-A,A]^N}
      |g(t,x-y)| \rmd y \, \rmd t .
\end{align*}
For any $q> p_0$, there is some $r>1$  such that
\[
  \frac{1}{r} + \frac{1}{p_0} = \frac{1}{q}+1.
\]
Observe that
\begin{align}
\|g(t,\cdot)\|_{L^{p_0}} = | \varphi(t)| \cdot
    \Big\|  \int_{[0,1]^N} f(\cdot +ty)\rmd y  \Big\|_{L^{p_0}}
\lesssim \| \varphi \|_{L^{\infty}}\cdot \|f\|_{L^{p_0}}.
\label{eq:s:21}
\end{align}
We see from Young's inequality that
\begin{align*}
\|F\|_{L^{q}}
&\lesssim \|f\|_{L^{p_0}} .
\end{align*}
Setting $q=p_1$, we get $\|F\|_{L^{p_1}}\lesssim \|f\|_{L^{p_0}}$.

Combing results in (S1), (S2), and (S3), yields
\[
  W_s^{(p_0,p_1)}\hookrightarrow W_s^{(p_1,p_1)} = W_s^{p_1}.
\]

For the case $p_0>p_1$, we provide a counterexample in Example~\ref{Ex:ex1}.
This completes the proof.
\end{proof}

The following is an immediate consequence.

\begin{Corollary}\label{Co:embedding:s}
Suppose that $\vec p=(p_0,p_1)$
with $1\le p_0\le p_1<\infty$.
For any $0<\tilde s< s < 1$, we have
\begin{equation}\label{eq:Ws s}
W_1^{\vec p} \hookrightarrow  W_{s}^{\vec p}  \hookrightarrow W_{\tilde s}^{\vec p}
 .
\end{equation}
\end{Corollary}

\begin{proof}
First, we prove that $W_1^{\vec p} \hookrightarrow  W_{s}^{\vec p}$.
Take some $f\in W_1^{\vec p}$.  As  for functions in classical Sobolev spaces,
for almost all $x\in\bbR^N$, we have
\[
  |f(x) - f(x+ae_j)  |^{p_1} = \Big|\int_0^a \partial_{x_j} f(x+te_j)\rmd t\Big|^{p_1}
   \le |a|^{p_1-1}   \int_0^a |\partial_{x_j} f(x+te_j)|^{p_1}\rmd t .
\]
Hence
\begin{align*}
\int_{\bbR^N} |f(x) - f(x+ae_j)  |^{p_1} \rmd x \le |a|^{p_1}
  \|\partial_{x_j}f\|_{L^{p_1}}^{p_1}.
\end{align*}
Therefore,
\begin{align*}
\int_{|a|<1}
  \int_{\bbR^N} \frac{|f(x) - f(x+ae_j)  |^{p_1}}{|a|^{1+sp_1}} \rmd x \, \rmd a
  \le \int_{|a|<1} \frac{|a|^{p_1}}{|a|^{1+sp_1}}  \|\partial_{x_j}f\|_{L^{p_1}}^{p_1} \rmd a
  \lesssim \|\partial_{x_j}f\|_{L^{p_1}}^{p_1}.
\end{align*}
On the other hand, since $p_0\le p_1$,
we have $W_1^{\vec p} \hookrightarrow  W_1^{p_1}$, thanks to Lemma~\ref{Lm:L3a}. Hence
\begin{align*}
\int_{|a|\ge 1}
  \int_{\bbR^N} \frac{|f(x) - f(x+ae_j)  |^{p_1}}{|a|^{1+sp_1}} \rmd x \, \rmd a
  \lesssim
    \int_{|a|\ge 1} \frac{\|f\|_{L^{p_1}}^{p_1} }{|a|^{1+sp_1}}  \rmd a
  \lesssim \|f\|_{L^{p_1}}^{p_1}.
\end{align*}
Consequently, $\|f\|_{W_s^{\vec p}} \lesssim \|f\|_{W_1^{\vec p}}$.

Next, we show that $W_s^{\vec p} \hookrightarrow  W_{\tilde s}^{\vec p}$.
Fix some $f\in W_s^{\vec p}$.
Since $\tilde s<s$, we have
\begin{align*}
\int_{|a|<1}
   \int_{\bbR^N} \frac{|f(x)-f(x+ae_j)|^{p_1}}{a^{1+p_1 \tilde s}} \rmd x \, \rmd a
\le
\int_{|a|<1}
   \int_{\bbR^N} \frac{|f(x)-f(x+ae_j)|^{p_1}}{a^{1+p_1  s}} \rmd x \, \rmd a
 \lesssim [f]_{W_s^{p_1}}^{p_1}.
\end{align*}
On the other hand, by Lemma~\ref{Lm:L3a}, $f\in L^{p_1}$.
Hence
\begin{align*}
\int_{|a|\ge 1}
   \int_{\bbR^N} \frac{|f(x)-f(x+ae_j)|^{p_1}}{a^{1+p_1 \tilde s}} \rmd x \, \rmd a
\le
\int_{|a|\ge 1}
    \frac{C_{p_1} \|f\|_{L^{p_1}}^{p_1}}{a^{1+p_1  \tilde s}}  \rmd a
 \lesssim \|f\|_{L^{p_1}}^{p_1}.
\end{align*}
This completes the proof.
\end{proof}

We point out that the inclusion in Corollary~\ref{Co:embedding:s}
is not true in general for the case $p_0>p_1$,
which is quite different from the classical case, since
for any $1\le p<\infty$ and $0<\tilde s< s<1$,
\[
    W_1^p\hookrightarrow
   W_s^p \hookrightarrow W_{\tilde s}^p .
\]
Below is a counterexample.

\begin{Example}\label{Ex:ex1}
Suppose that $0<s\le 1$, $1\le p_0, p_1 <\infty$ and $p_1\delta<N<p_0\delta$.
Let $f(x) =  (1+ |x|^2)^{-\delta/2}$, $x\in\bbR^N$
 and let $\vec p = (p_0,p_1)$.
We have
 $f\in W_s^{\vec p}(\bbR^N)$ if and only if
 $p_1(\delta+s)> N$.
Consequently, for any  $0<s\le 1$ and $p_0>p_1$,
$W_s^{\vec p} \not\subset W_s^{p_1}$.

Moreover, if   $0<\tilde s  < s\le 1$ and $\tilde s<N/p_1 - N/p_0$,
we have  $ W_s^{\vec p} \not\subset
 W_{\tilde s}^{\vec p}$.
\end{Example}

\begin{proof}
(i)\, We show that $f\in W_s^{\vec p}(\bbR^N)$ if and only if
  $p_1(\delta+s)> N$,
 for which we only need to consider the case $0<s<1$ since
 the other case $s=1$ is obvious.

First, we assume that   $p_1(\delta+s)> N$. Let us estimate the integral
\[
 \int_{\bbR^N}  |f(x)-f(x+ae_j)|^{p_1}  \rmd x,\qquad 1\le j\le N.
\]
Denote $\hat x_j = x-x_je_j$
and $\rmd \hat x_j = \prod_{i\ne j} \rmd x_i$. Observe that
\begin{align}
  |f(x) - f(x+ae_j)|^{p_1}
  = \Big|\int_0^a  \partial_{x_j} f(x+te_j)\rmd t\Big|^{p_1}
  \le |a|^{p_1-1} \Big|\int_0^a  |\partial_{x_j}f(x+te_j)|^{p_1}\rmd t\Big|.  \label{eq:s3:e1}
\end{align}
When $|a|>1$, $|t|\le |a|$ and $|x|>2|a|$, we have  $|x+te_j| >|a|$. Hence
\begin{align}
&  \int_{|x|>2|a| }
 |f(x) - f(x+ae_j)|^{p_1}\rmd x  \nonumber \\
&\lesssim  |a|^{p_1-1}\bigg|\int_0^a \int_{|x|>2|a|}
      \frac{|x_j+t|^{p_1}}
      {(1+|\hat x_j|^2 + |x_j+t|^2)^{(\delta+2)p_1/2}}\rmd x  \, \rmd t
       \bigg|
    \nonumber  \\
&\le   |a|^{p_1-1}\bigg|\int_0^a   \int_{|x|>|a|}
      \frac{|x_j |^{p_1}}
      {(1+|\hat x_j|^2+|x_j|^2  )^{(\delta+2)p_1/2}}
      \rmd x  \, \rmd t\bigg|
    \nonumber  \\
&=   |a|^{p_1}\bigg( \int_{|x_j|>|a|}
     \int_{\bbR^{N-1}}
      \frac{|x_j |^{p_1}}
      {(1+|\hat x_j|^2+|x_j|^2  )^{(\delta+2)p_1/2}}\rmd \hat x_j  \, \rmd x_j
    \nonumber  \\
&\qquad +   \int_{|x_j|<|a|}
     \int_{|\hat x_j|> (a^2 - x_j^2)^{1/2}}
      \frac{|x_j |^{p_1}}
      {(1+|\hat x_j|^2+|x_j|^2  )^{(\delta+2)p_1/2}}\rmd \hat x_j  \, \rmd x_j\bigg)
    \nonumber  \\
&\approx   |a|^{p_1} \bigg(\int_{|x_j|>|a|}
       \frac{|x_j |^{p_1}\rmd x_j}
      {(1+|x_j|  )^{(\delta+2)p_1-(N-1)}}
    + \int_{|x_j|<|a|}
    \frac{|x_j |^{p_1}\rmd x_j}
      {(1+|a|  )^{(\delta+2)p_1-(N-1)}} \bigg)
    \nonumber  \\
&\approx  |a|^{N-p_1\delta } ,
     \label{eq:s:8}
\end{align}
where we use the fact that $(\delta+1)p_1 > (\delta+s)p_1>N$.
On the other hand,
\begin{align}
&  \int_{ |x|\le 2|a| }   |f(x)-f(x+ae_j)|^{p_1} \rmd x
  \nonumber \\
&\lesssim \int_{|x|\le 2|a|} \Big(\frac{1}{(1+ |x|^2)^{p_1\delta/2}}
    +\frac{1}{(1+ |x +ae_j|^2)^{p_1\delta/2}} \Big) \rmd x
      \nonumber \\
&\le \int_{|x|\le 2|a|} \frac{1}{(1+ |x|^2)^{p_1\delta/2}}\rmd x
    +\int_{|x+ae_j|\le 3|a|} \frac{1}{(1+ |x +ae_j| ^2)^{p_1\delta/2}}  \rmd x
      \nonumber \\
&\lesssim |a|^{N-p_1\delta}.   \label{eq:s:9}
\end{align}
Putting (\ref{eq:s:8}) and  (\ref{eq:s:9}) together, we get
\[
  \int_{\bbR^N}  |f(x)-f(x+ae_j)|^{p_1}  \rmd x
  \lesssim |a|^{N- p_1\delta},\quad |a|>1.
\]

For $|a|\le 1$, since $|\partial_{x_j} f(x)|\lesssim
1/(1+|x|^2)^{(\delta+1)/2}$ and
$p_1(\delta+1) > p_1(\delta+s)>N$,
we have $\partial_{x_j} f \in L^{p_1}$.
It follows from  (\ref{eq:s3:e1}) that
\[
  \int_{\bbR^N}|f(x) - f(x+ae_j)|^{p_1}\rmd x
  \le |a|^{p_1} \|\partial_{x_j} f\|_{L^{p_1}}^{p_1}.
\]
Combining the preceding estimates  we   deduce
\begin{align*}
\int_{\bbR} \int_{\bbR^N} \frac{|f(x)-f(x+ae_j)|^{p_1}}{|a|^{1+p_1s}} \rmd x \,\rmd a
&=\int_{|a|\le 1}  \int_{\bbR^N} \frac{|f(x)-f(x+ae_j)|^{p_1}}{|a|^{1+p_1s}} \rmd x  \,\rmd a   \\
  & \qquad
   + \int_{|a|>1}  \int_{\bbR^N} \frac{|f(x)-f(x+ae_j)|^{p_1}}{|a|^{1+p_1s}} \rmd x  \,\rmd a  \\
&\lesssim \int_{|a|\le 1}  |a|^{p_1-1-p_1s} \rmd a +
   \int_{|a|> 1} \frac{\rmd a}{|a|^{p_1(\delta+s)-N+1}}  \\
&<\infty .
\end{align*}
Hence $f\in W_s^{\vec p}$.

Next we show that $f\not\in W_s^{\vec p}$ whenever $p_1(\delta+s)\le N$.

When $a>2$, $a^2 < |x|^2 < 5a^2/4$ and $x_j>0$, we have $1+ |x|^2 < 3a^2/2$
and $ 1 +  | x+ae_j|^2 = 1+ |x|^2 + 2ax_j + a^2 > 2a^2$.
Hence
\begin{align*}
\int_{\bbR^N} |f(x) - f(x+ae_j)|^{p_1} \rmd x
& \ge \int_{\substack{a^2 < |x|^2 < 5a^2/4 \\ x_j>0}} |f(x) - f(x+ae_j)|^{p_1} \rmd x
  \\
& \gtrsim \int_{\substack{a^2 < |x|^2 < 5a^2/4 \\ x_j>0}}   \frac{\rmd x }{a^{p_1\delta}} \\
&\approx a^{N-p_1 \delta}.
\end{align*}
Therefore,
\begin{align*}
\int_{\bbR}  \int_{\bbR^N} \frac{|f(x)-f(x+ae_j)|^{p_1}}{|a|^{1+p_1s}} \rmd x\,\rmd a
&\gtrsim \int_{a\ge 2} \frac{ \rmd a }{a^{p_1(\delta+s)-N+1}}
=\infty .
\end{align*}

\medskip\noindent
(ii)\, Now assume that $p_0>p_1$. Take some $\varepsilon>0$ such that
$\varepsilon< \min\{s,  N/p_1-N/p_0\}$. Set $\delta = N/p_1 - \varepsilon$.
We have $  p_1\delta<N<p_0\delta$
 and $p_1(\delta+s) = N + p_1(s-\varepsilon) >N$.
Hence $f\in W_{s}^{\vec p} \setminus W_s^{p_1}$.

\medskip\noindent
(iii)\,
When
$\tilde s<N/p_1 - N/p_0$,
we have $p_0>p_1$.
Take some constant $\delta$ such that
\[
  \max\Big\{\frac{N}{p_0}, \frac{N}{p_1} -s  \Big\} < \delta < \frac{N}{p_1} -\tilde s.
\]
Then $ p_0\delta>N>p_1\delta$, $p_1(\delta+s)>N$   and $p_1(\delta+\tilde s)<N$.
Hence $f\in W_s^{\vec p} \setminus W_{\tilde s}^{\vec p}$.
\end{proof}

\begin{proof}[Proof of Theorem~\ref{thm:embedding:fractional}]
For the case $p_0\!\le p_1$, the conclusion follows from
Lemma~\ref{Lm:L3a} and the embedding theorem for
classical fractional Sobolev spaces.

Now we assume that $p_0>p_1$.
It suffices to show that if $f\in W_s^{\vec p}$,
then
$f\in L^q$ with $q = Np_1/(N-sp_1)$.

We adopt the notation used in the proof of Lemma~\ref{Lm:L3a}.
We rewrite (\ref{eq:s:e17}) as
\begin{align}
g &=    ((\nabla_x \Phi) E) * \nabla_x g  +(\Phi \nabla_x  E) * \nabla_x g
        + ((\partial_t \Phi) E) * h_0 + ( \Phi \partial_t  E) * h_0
   \nonumber \\
&\qquad
        + ((\partial_t \Phi) E) * h_1
        + (\Phi \partial_t  E ) * h_1
      -2 (\nabla \Phi \cdot \nabla E) * g
        - ((\Delta\Phi)  E) *g
          .\label{eq:s:e17a}
\end{align}
Checking the proof of
\cite[Proposition 4.47]{Demengel2012}, we find
that
\begin{equation}\label{eq:s:20}
  \Big\|\Big( ((\nabla_x \Phi) E) * \nabla_x g  +(\Phi \nabla_x  E) * \nabla_x g
        +( (\partial_t \Phi) E )* h_0 +  (\Phi \partial_t  E) * h_0\Big)(0,\cdot)\Big\|_{L^q}
        \lesssim \|f\|_{W_s^{\vec p}}.
\end{equation}
In fact,  the only property of
$g$ used in the proof  of
\cite[Proposition 4.47]{Demengel2012}
is that $ t^{\alpha}\nabla g(t,x)$ lies in $L^{p_1}$.
In our case, $\partial_t g$ is replaced by $h_0$.
By Lemma~\ref{Lm:Ls1}, $t^{\alpha}h_0(t,x)\in L^{p_1}$.
So (\ref{eq:s:20}) is true.

It remains to consider the last four terms in
(\ref{eq:s:e17a}).
Since $\Phi$ is compactly supported, we have
\[
  |(\partial_t \Phi) E(t,x)|
  \lesssim \frac{1}{(|t|+|x|)^N},
\]
and
\[
 | \Phi \partial_t  E (t,x)| \lesssim \frac{1}{(|t|+|x|)^N}.
\]
Observe that
$h_1(t,x)=0$ for $t<1$ or $t>A$. We have
\begin{align*}
\Big|\Big(((\partial_t \Phi) E )* h_1
        + (\Phi \partial_t  E )* h_1\Big)(0,x)\Big|
\lesssim  \int_1^A  \int_{[-A,A]^N}
           \frac{|h_1(t,x-y)|}{(t+|y|)^N} \rmd y  \, \rmd t.
\end{align*}
Take some $r>1$ such that $1/p_0 + 1/r = 1/q +1$.
Applying Young's inequality, we deduce from (\ref{eq:s:11c})
that
\[
  \Big\|\Big(((\partial_t \Phi) E) * h_1
        + (\Phi \partial_t  E) * h_1\Big)(0,\cdot)\Big\|_{L^q}
  \lesssim \int_1^A
           \|h_1(t,\cdot)\|_{L^{p_0}} \Big\|\frac{1}{(t+|\cdot|)^N}\Big\|_{L^r} \rmd t
  \lesssim \|f\|_{L^{p_0}}.
\]

For the last two terms in (\ref{eq:s:e17a}),
we show in (S3) of the proof of  Lemma~\ref{Lm:L3a}
that
\[
  \Big\|\Big( |(\nabla \Phi \cdot \nabla E) *g|
        +| ((\Delta\Phi) E )*g|
        \Big)(0,\cdot)\Big\|_{L^q}
  \lesssim \|f\|_{L^{p_0}}.
\]
Combining these facts we derive the desired conclusion.
\end{proof}

\subsection{Density of Compactly Supported Infinitely Differentiable Functions}

In this subsection, we show that compactly supported
infinitely differentiable functions are dense
in fractional nonuniform Sobolev spaces for certain indices.

First, we show that  smooth functions are
dense in nonuniform fractional Sobolev spaces.

\begin{Lemma}\label{Lm:frac:density}
For any $0<s<1$ and $\vec p = (p_0,p_1)$ with
$1\le p_0,p_1<\infty$,
$C^{\infty}\cap W_s^{\vec p}(\bbR^{N})$
is dense in $W_s^{\vec p}(\bbR^{N})$.
\end{Lemma}

\begin{proof}
As in the proof of Lemma~\ref{Lm:density},
take some
$\varphi\in C_c^{\infty}$ such that
$\varphi$ is nonnegative and $\|\varphi\|_{L^1}=1$.
Fix some $f\in  W_s^{\vec p}$ and set
\[
  g_{\lambda} (x) = \int_{\bbR^N} f(y) \frac{1}{\lambda^N}\varphi(\frac{x-y}{\lambda})
    \rmd y, \qquad \lambda>0.
\]
We have $g_{\lambda}\in C^{\infty}$.
Moreover, since
\[
  g_{\lambda} (x) = \int_{\bbR^N} f(x-\lambda y) \varphi(y)
    \rmd y,
\]
we have
\[
  g_{\lambda} (x) -f(x)= \int_{\bbR^N} (f(x-\lambda y)-f(x)) \varphi(y)
    \rmd y.
\]
Now
we derive from Minkowski's inequality
and the continuity of translation operators in
Lebesgue spaces
 that
\begin{equation}\label{eq:a:e3}
  \lim_{\lambda\rightarrow 0}
     \|g_{\lambda} - f \|_{L^{p_0}}
 \le
  \lim_{\lambda\rightarrow 0}
      \int_{\bbR^N} \| f(\cdot -\lambda y) -f\|_{L^{p_0}}\varphi(y)
    \rmd y
     =0
\end{equation}
and
\begin{align*}
 & \lim_{\lambda\rightarrow 0}
   [g_{\lambda} -f]_{W_s^{p_1}}\\
&\le \lim_{\lambda\rightarrow 0}\int_{\bbR^N}    \bigg \| \frac{\big(f(x-\lambda y) - f(x)\big) - \big(f(z-\lambda y) - f(z)\big)}
    {|x-z|^{N/p_1+s}}
    \bigg\|_{L_{(x,z)}^{p_1}(\bbR^N\times\bbR^N)}
    \varphi(y)
    \rmd y  \\
&= \lim_{\lambda\rightarrow 0}\int_{\bbR^N}    \bigg \| \frac{f(x-\lambda y) - f(z-\lambda y)}
    {|(x-\lambda y)-(z-\lambda y)|^{N/p_1+s}}
-\frac{f(x) - f(z)}
    {|x -z|^{N/p_1+s}}
    \bigg\|_{L_{(x,z)}^{p_1}(\bbR^N\times\bbR^N)}
    \varphi(y)
    \rmd y  \\
&=0,
\end{align*}
having applied the fact that
$(f(x) - f(z))/
    |x -z|^{N/p_1+s} \in L_{(x,z)}^{p_1}$.
Hence $\lim_{\lambda\rightarrow 0}\|g_{\lambda} - f\|_{W_s^{\vec p}}=0$.
\end{proof}

As in the classical case,
to prove the density of compactly supported smooth functions,
we first approximate
a function in fractional nonuniform Sobolev spaces
by its  truncation,
then approximate the truncation by  its regularization.
However, since $p_0$ needs not to be identical to $p_1$,
the technique details are quite different.

\begin{Theorem}\label{dense:frac}
Suppose that $0<s<1$, $\vec p = (p_0,p_1)\in [1,\infty)^2$ and $s/N \ge 1/p_1 - 1/p_0$.
Then $C_c^{\infty}(\bbR^N)$ is dense in
$W_s^{\vec p}(\bbR^N)$.
\end{Theorem}

\begin{proof}
Take some nonnegative function
$\varphi\in C_c^{\infty}$ such that
$\varphi(x)=1$ when $|x|<1$
and $\varphi(x)=0$ when $|x|>2$.

Fix some $f\in W_s^{\vec p}$.
First, we show that $\varphi(\cdot/n)f$ tends to $f$ in $W_s^{\vec p} $
as $n$ tends to the infinity.
Since $\varphi(\cdot/n)f$ tends to $f$ in $L^{p_0}$, thanks to Lebesgue's dominated convergence theorem, it suffices to show that
$\lim_{n\rightarrow\infty} [\varphi(\cdot/n)f-f]_{W_s^{p_1}}=0$.

If $p_0\le p_1$, by Lemma~\ref{Lm:L3a}, we have
$W_s^{\vec p}\hookrightarrow W_s^{p_1}$.
Now we see from the density result in
the classical fractional Sobolev spaces that
$\lim_{n\rightarrow\infty} [\varphi(\cdot/n)f-f]_{W_s^{p_1}}=0$.

It remains to consider the case $p_0>p_1$.
Observe that
\begin{align*}
\Big(\varphi\big(\frac{x}{n}\big)-1\Big) f(x)
-\Big(\varphi\big(\frac{y}{n}\big)-1\Big) f(y)
= \Big(\varphi\big(\frac{x}{n}\big)-1\Big) (f(x)-f(y))
   + \Big(\varphi\big(\frac{x}{n}\big)-\varphi\big(\frac{y}{n}\big)\Big) f(y).
\end{align*}
We have
\begin{align*}
[\varphi\big(\frac{\cdot}{n}\big)f -f]_{W_s^{p_1}}^{p_1}
&\lesssim
  \iint_{\bbR^{2N}} \frac{|(\varphi(x/n)-1)(f(x)-f(y))|^{p_1}}
    {|x-y|^{N+sp_1}} \rmd x\,\rmd y
  \\
&\qquad
  +
  \iint_{\bbR^{2N}} \frac{|(\varphi(x/n)-\varphi(y/n))f(y)|^{p_1}}
    {|x-y|^{N+sp_1}} \rmd x\,\rmd y.
\end{align*}
Since $f\in W_s^{\vec p}$, we see from Lebesgue's dominated convergence theorem
that
\[
  \lim_{n\rightarrow\infty }
  \iint_{\bbR^{2N}} \frac{|(\varphi(x/n)-1)(f(x)-f(y))|^{p_1}}
    {|x-y|^{N+sp_1}} \rmd x\,\rmd y  = 0.
\]
Now we only need to show that
\begin{align}\label{eq:s32:e1}
 \lim_{n\rightarrow\infty }
 \iint_{\bbR^{2N}} \frac{|(\varphi(x/n)-\varphi(y/n))f(y)|^{p_1}}
    {|x-y|^{N+sp_1}} \rmd x\,\rmd y=0.
\end{align}
We split the integral into two parts,
\begin{align*}
 & \iint_{\bbR^{2N}} \frac{|(\varphi(x/n)-\varphi(y/n))f(y)|^{p_1}}
    {|x-y|^{N+sp_1}} \rmd x\,\rmd y \\
 &=\Big(\iint_{|y|\le 4n} +  \iint_{|y|> 4n}  \Big)  \frac{|(\varphi(x/n)-\varphi(y/n))f(y)|^{p_1}}
    {|x-y|^{N+sp_1}} \rmd x\,\rmd y \\
&= I+II.
\end{align*}
We have
\begin{align*}
I&= \Big(\iint_{\substack{|y|\le 4n \\ |x-y|\le 8n}} +
  \iint_{\substack{|y|\le 4n \\ |x-y|> 8n}}\Big)  \frac{|(\varphi(x/n)-\varphi(y/n))f(y)|^{p_1}}
    {|x-y|^{N+sp_1}} \rmd x\,\rmd y  \\
&=I_1 + I_2,
\end{align*}
where
\begin{align*}
I_1
&\le \frac{\|\nabla\varphi\|_{L^{\infty}}^{p_1}}
  {n^{p_1}}\iint_{\substack{|y|\le 4n \\ |x-y|\le 8n}}  \frac{|x-y|^{p_1}  | f(y)|^{p_1}}
    {|x-y|^{N+sp_1}} \rmd x\,\rmd y
\lesssim \frac{1}{n^{sp_1}} \int_{|y|\le 4n}  |f(y)|^{p_1}\rmd y,
\end{align*}
and
\begin{align*}
I_2
&= \iint_{\substack{|y|\le 4n \\ |x-y|> 8n}}  \frac{ |\varphi(y/n) f(y)|^{p_1}}
    {|x-y|^{N+sp_1}} \rmd x\,\rmd y
 \lesssim \frac{1}{n^{sp_1}} \int_{|y|\le 4n}  |f(y)|^{p_1}\rmd y.
\end{align*}
Hence
\[
  I \lesssim  \frac{1}{n^{sp_1}} \int_{|y|\le 4n}  |f(y)|^{p_1}\rmd y.
\]

Denote $r = p_0/p_1$. For any $k\ge 1$, there is some $n_0$ such that
\[
  \| f \cdot 1_{\{|y|\le n_0\}}\|_{L^{p_0}} \ge \Big(1 -\frac{1}{k}\Big) \|f\|_{L^{p_0}}.
\]
Applying H\"older's inequality, we get
\begin{align*}
  I &\lesssim  \frac{1}{n^{sp_1}}\Big( \int_{|y|\le n_0}  |f(y)|^{p_1}\rmd y
      +  \int_{n_0<|y|\le 4n}  |f(y)|^{p_1}\rmd y \Big)\\
    &\lesssim \frac{1}{n^{sp_1}}\Big( n_0^{N/r'} \| |f|^{p_1}\|_{L^{r}}
      + n^{N/r'} \| |f|^{p_1} \cdot 1_{\{|y|>n_0\}}\|_{L^{r}}  \Big) \\
 &\le  \frac{n_0^{N/r'}}{n^{sp_1}}  \| f\|_{L^{p_0}}^{p_1}
    + \frac{1}{k^{p_1}} \cdot \frac{1}{n^{Np_1(s/N-1/p_1+1/p_0)}}\| f\|_{L^{p_0}}^{p_1}.
\end{align*}
Since $s/N\ge 1/p_1 - 1/p_0$, letting $n\rightarrow\infty$ and $k\rightarrow\infty$ successively,
we get
\[
  \lim_{n\rightarrow \infty} I = 0.
\]

Next we deal with $II$. Recall that $\varphi(x)=0$ when $|x|>2$.
We have
\begin{align*}
II&= \iint_{|y|> 4n}   \frac{|(\varphi(x/n)-\varphi(y/n))f(y)|^{p_1}}
    {|x-y|^{N+sp_1}} \rmd x\rmd y \\
&= \int_{|x|\le 2n} \int_{|y|> 4n}   \frac{|\varphi(x/n)f(y)|^{p_1}}
    {|x-y|^{N+sp_1}} \rmd x\rmd y \\
&\lesssim  \int_{|x|\le 2n} |\varphi(\frac{x}{n})|^{p_1}
  \rmd x \int_{|y|> 4n}   \frac{|f(y)|^{p_1}}
    {|y|^{N+sp_1}} \rmd y \\
&\lesssim n^{N}\cdot \|f\cdot 1_{\{|y|\ge 4n\}} \|_{L^{p_0}}^{p_1}
  \cdot \frac{1}{n^{sp_1+N/r}}   \\
&=   \frac{1}{n^{Np_1(s/N-1/p_1+1/p_0)}}  \|f\cdot 1_{\{|y|\ge 4n\}} \|_{L^{p_0}}^{p_1}.
\end{align*}
Hence $\lim_{n\rightarrow \infty} II = 0$.
Therefore, (\ref{eq:s32:e1}) is true.

The above arguments show that the sequence $
\{\varphi(\cdot/n)f\:\, n\ge 1\}$ is a subset of $W_s^{\vec p}$
and is convergent to $f$ in $W_s^{\vec p}$.

To finish the proof, it suffices to show that for any compactly supported function
$f\in W_s^{\vec p}$, $f$ can be approximated by functions in $C_c^{\infty}$,
which can be achieved with the same arguments as in the proof of Lemma~\ref{Lm:frac:density}.
\end{proof}

\section{Applications}
 \label{sec:Examples}

\subsection{Local estimates for solutions of  heat equations}

Consider the classical solution of the heat equation
\[
  \begin{cases}
\partial_t u(t,x) - \Delta_x u(t,x) = 0,& \quad t>0, \\
u(0,x) = u_0(x).&
\end{cases}
\]
That is,
\begin{equation}\label{eq:u:heat}
  u(t,x) =
     \frac{1}{(4\pi t)^{N/2}} \int_{\bbR^N}
        e^{-|x-y|^2/(4t)} u_0(y) \rmd y, \qquad t>0, x\in\bbR^N.
\end{equation}
Fefferman,  McCormick, Robinson, and Rodrigo \cite{FeffermanMcCormickRobinsonRodrigo2017}
studied local energy estimates for $u(t,x)$.
They proved that if the initial data
$u_0$ belongs to the Sobolev space
$H^s$, then for any $T>0$,
the classical solution of the heat equation
satisfies that
$u
\in L^{\infty}(0,T; H^s) \cap L^2(0,T; H^{s+1})$
and $t^{1/2}u(t,x)\in  L^2(0,T; H^{s+2})$.
As a result, $u\in
L^q(0,T; H^{s+2})$
for any $0<q<1$.

In this subsection, we show that the local estimates for
solutions of   heat equations
are also valid when the initial data belong   to
nonuniform Sobolev spaces.

\begin{Theorem}\label{thm:HeatEquation}
Suppose that  $s>0$, $T>0$ and
$\vec p = (p_0, \ldots, p_{\lceil s\rceil})$
with $1< p_i<\infty$ for $0\le i\le \lceil s\rceil$.
Set $\vec r = (p_0,\ldots, p_{\lfloor s\rfloor},2 )$
for $s\not\in \bbZ$
and $\vec r = \vec p$ for $s\in\bbZ$.
Let $u$ be the classical solution of the heat equation
with initial data $u_0\in W_s^{\vec p}(\bbR^N)$.
Then $u\in L^{\infty}(0,T; W_s^{\vec p}(\bbR^N))$
with
\begin{equation}\label{eq:u:e0a}
\sup_{0\le t\le T} \|u(t,\cdot)\|_{W_s^{\vec p}}
  \le  C_{s,\vec p,N}\|u_0\|_{W_s^{\vec p}}.
\end{equation}
Moreover, if $1<p_{\lfloor s\rfloor},
p_{\lceil s\rceil}\le 2$, then
\begin{align}
\int_0^T t^{\varrho}\|u(t,\cdot)\|_{W_{s+1}^{(\vec r, 2)}}^2 \rmd t
  &\le C_{s,\vec p,N} (1+T^{\theta_1})\|u_0\|_{W_s^{\vec p}}^2,
  \label{eq:u:e0b}
\\
  \int_0^T t^{1+\varrho}\|u(t,\cdot)\|_{W_{s+2}^{(\vec r, 2, 2)}}^2 \rmd t
  &\le C_{s,\vec p,N} (1+T^{\theta_2})\|u_0\|_{W_s^{\vec p}}^2,\label{eq:u:e0c}
\end{align}
 where
\begin{align*}
  \varrho = (2-p_s)\sigma,\quad
  p_s  = \min\{p_{\lfloor s\rfloor},  p_{\lceil s\rceil}
    \},
    \quad
 \sigma = \frac{N}{2p_{\lfloor s\rfloor}}+\frac{1}{2},
\end{align*}
and $\theta_1, \theta_2$ are constants.
 Consequently, for any $0<q<2/(2+\varrho)$,
$u \in L^q(0,T; W_{s+2}^{(\vec r,2,2)}(\bbR^N))$
 and
\begin{equation}\label{eq:u:q}
  \int_0^T \|u(t,\cdot)\|_{W_{s+2}^{(\vec r, 2,2)}}^q
  \rmd t
  \le C_{s,\vec p, N} T^{1-(1+\varrho/2)q}  (1+T^{\theta_2})^{q/2}
    \|u_0\|_{W_s^{\vec p}}^q.
\end{equation}
\end{Theorem}

Before giving a  proof of  Theorem~\ref{thm:HeatEquation},
we discuss a lemma.

\begin{Lemma}\label{Lm:Lp12}
Suppose that $f\in C_b^2(\bbR)$ and $f, f''\in L^p(\bbR)$ for some
$1<p<\infty$. Then $f'' |f|^{p-2}f\in L^1(\bbR)$
and
\[
  \int_{\bbR} f''(x) |f(x)|^{p-2} f(x) \rmd x
  = - (p-1)\int_{\bbR} |f'(x)|^2 |f(x)|^{p-2} \rmd x,
\]
where we apply the convention
that $|f'(x)|^2 |f(x)|^{p-2}=0$ whenever
$f'(x)=0$.
\end{Lemma}

\begin{proof}
Since $f, f''\in L^p$, we see from H\"older's
inequality that
$ f'' |f |^{p-2} f \in L^1$.
To prove the conclusion, it suffices to show
the following equations,
\begin{align}
  \int_0^{\infty} f''(x) |f(x)|^{p-2} f(x) \rmd x
  &=-f'(0)|f(0)|^{p-2} f(0) - (p-1)\int_0^{\infty} |f'(x)|^2 |f(x)|^{p-2} \rmd x, \label{eq:u:e21} \\
  \int_{-\infty}^0 f''(x) |f(x)|^{p-2} f(x) \rmd x
  &= f'(0)|f(0)|^{p-2} f(0)- (p-1)\int_{-\infty}^0 |f'(x)|^2 |f(x)|^{p-2} \rmd x.\label{eq:u:e22}
\end{align}
We prove only the first equation, and the second
one can be proved similarly.

Since $f\in L^p(\bbR)$, there exists
a sequence  $\{x_n:\, n\ge 1\}\subset (0,\infty)$
such that
\begin{equation}\label{eq:xn}
  \lim_{n\rightarrow\infty} x_n =  \infty
  \mbox{\quad and\quad}
  \lim_{n\rightarrow\infty} f(x_n)
= 0.
\end{equation}
When $p\ge 2$, $|f(x)|^{p-2} f(x)$ is continuously differentiable.    Integrating by parts
we obtain
\begin{align*}
&\int_0^{\infty} f''(x) |f(x)|^{p-2} f(x) \rmd x
  \\
&= \lim_{n\rightarrow\infty}
    \int_0^{x_n} f''(x) |f(x)|^{p-2} f(x) \rmd x
    \\
&=  \lim_{n\rightarrow\infty}
    \Big( f'(x_n) |f(x_n)|^{p-2}f(x_n)
     - f'(0)|f(0)|^{p-2}f(0)
    - (p-1)\int_0^{x_n} |f'(x)|^2 |f(x)|^{p-2} \rmd x
     \Big) \\
&=   - f'(0)|f(0)|^{p-2}f(0)
    - (p-1)\int_0^{\infty} |f'(x)|^2 |f(x)|^{p-2}
     \rmd x .
\end{align*}
Hence (\ref{eq:u:e21}) is valid.

It remains to consider the case $1<p<2$.
Since $f$ is continuous, the set $E := \{x>0:\, f(x)\ne 0\}$
is   open   in $(0,\infty)$.
Consequently, $E$
is the union of at most countable
pairwise disjoint intervals $(a_i, b_i)$, $i\in I $.
There are three cases:

\medskip\noindent
(i)\, $b_i<\infty$ for each $i\in I$.
\medskip

For each interval $(a_i,b_i)$
and $\varepsilon>0$ small enough,
$|f(x)|^{p-2} f(x) $ is continuously differentiable
on $[a_i+\varepsilon, b_i-\varepsilon]$.
Integrating by parts again, we write
\begin{align}
&\int_{a_i}^{b_i}
       f''(x) |f(x)|^{p-2} f(x) \rmd x  \nonumber \\
 &=\lim_{\varepsilon\rightarrow 0}
     \int_{a_i+\varepsilon}^{b_i-\varepsilon}
             f''(x) |f(x)|^{p-2} f(x) \rmd x
      \nonumber  \\
 &= \lim_{\varepsilon\rightarrow0}
     \Big(
     (f'|f|^{p-2}f) (b_i-\varepsilon)
     - (f'|f|^{p-2}f) (a_i+\varepsilon)
     - (p-1)\int_{a_i+\varepsilon}^{b_i-\varepsilon} |f'(x)|^2 |f(x)|^{p-2} \rmd x
     \Big)           \nonumber  \\
&= -(f'|f|^{p-2}f)(a_i)- (p-1)\int_{a_i }^{b_i } |f'(x)|^2 |f(x)|^{p-2} \rmd x,
  \label{eq:u:e24}
\end{align}
where $f(a_i)=0$ if $a_i\ne 0$.
Moreover, if $0\not\in \{a_i:\, i\in I\}$,
then $f(0)=0$.
It follows that
\begin{align}
 \int_0^{\infty} f''(x) |f(x)|^{p-2} f(x) \rmd x
&=   \sum_{i\in I}
    \int_{a_i}^{b_i}
       f''(x) |f(x)|^{p-2} f(x) \rmd x
   \nonumber  \\
&= - f'(0)|f(0)|^{p-2}f(0)
    -  \sum_{i\in I}
      (p-1) \int_{a_i}^{b_i} |f'(x)|^2 |f(x)|^{p-2} \rmd x .
      \label{eq:u:e23}
\end{align}

Next we show that the set
$F:=\{x>0:\, f(x)=0 \mbox { and } f'(x)\ne 0\}$
is at most countable.

Take some $x_0\in F$. If for any $\varepsilon>0$,
$((x_0-\varepsilon, x_0+\varepsilon)\cap F)
\setminus\{x_0\}\ne
\emptyset$, then there exists a sequence
$\{y_k:\, k\ge 1\}\subset F\setminus\{x_0 \}$ such that
$\lim_{k\rightarrow \infty} y_k=x_0$.
Consequently,
\[
  f'(x_0) =\lim_{k\rightarrow \infty}  \frac{f(y_k)-f(x_0)}{y_k-x_0} =0,
\]
which contradicts the fact $f'(x_0)\ne 0$.
Hence for any $x\in F$,
there is some
$\varepsilon>0$ such that
$(x-\varepsilon, x+\varepsilon)\cap F = \{x\}$.
Consequently,
there exist rational numbers $r_x$ and $R_x$
such that $r_x<x<R_x$
and $(r_x, R_x)\cap F = \{x\}$.
Since $\{(r_x,R_x):\, x\in F\}$ is at most countable,
$F = \cup_{x\in F} (r_x,R_x)\cap F$ is also at most countable.

Using the convention that $|f'(x)|^2 |f(x)|^{p-2}=0$
whenever $f'(x)=0$, we see from (\ref{eq:u:e23}) that
\begin{align}
\int_0^{\infty} f''(x) |f(x)|^{p-2} f(x) \rmd x
&= - f'(0)|f(0)|^{p-2}f(0)
    - (p-1)\int_{\{x:\, f(x)\ne 0\}}  |f'(x)|^2 |f(x)|^{p-2} \rmd x  \nonumber \\
 &\qquad
    - (p-1)\int_{\{x:\, f(x)=f'(x)=0\}}  |f'(x)|^2 |f(x)|^{p-2} \rmd x  \nonumber  \\
&=     - f'(0)|f(0)|^{p-2}f(0)
    -(p-1) \int_0^{\infty}  |f'(x)|^2 |f(x)|^{p-2} \rmd x.
    \nonumber
\end{align}
Hence (\ref{eq:u:e21}) is valid.

\medskip\noindent
(ii)   $b_{i_0}=\infty$ for some $i_0\in I$
and $a_{i_0}>0$.
\medskip

In this case, $b_i\le a_{i_0}$ for all $i\in I\setminus\{i_0\}$.
Note that
$
(f' |f|^{p-2}f)(a_{i_0})=0$.
Similar arguments as to those in Case (i) yield
that
\begin{align}
\int_0^{a_{i_0}} f''(x) |f(x)|^{p-2} f(x) \rmd x
=   - (f' |f|^{p-2}f)(0)
    -(p-1) \int_0^{a_{i_0}}  |f'(x)|^2 |f(x)|^{p-2} \rmd x.
     \label{eq:u:e25}
\end{align}

Recall that there is a sequence
 $\{x_n:\, n\ge 1\} $ satisfying (\ref{eq:xn}).
In analogy with   (\ref{eq:u:e24}) we obtain
\begin{align}
&\int_{a_{i_0}}^{\infty}
    f''(x) |f(x)|^{p-2} f(x) \rmd x \nonumber \\
&= \lim_{\substack{n\rightarrow\infty\\
   \varepsilon\rightarrow 0}}
   \int_{a_{i_0}+\varepsilon}^{x_n}   f''(x) |f(x)|^{p-2} f(x) \rmd x
  \nonumber  \\
&= \lim_{\substack{n\rightarrow\infty\\
   \varepsilon\rightarrow 0}}
   \Big((f' |f|^{p-2}f)  (x_n)
      - (f' |f|^{p-2}f)(a_{i_0}+\varepsilon)
      - (p-1)\int_{a_{i_0}}^{x_n}   |f'(x)|^2 |f(x)|^{p-2} \rmd x
       \Big)\nonumber \\
&=
      - (p-1)\int_{a_{i_0}}^{\infty}   |f'(x)|^2 |f(x)|^{p-2}
      \rmd x. \label{eq:u:e26}
\end{align}
Now (\ref{eq:u:e21}) follows from
(\ref{eq:u:e25}) and (\ref{eq:u:e26}).

\medskip\noindent
(iii) $E = (0,\infty)$.
\medskip

Similar arguments as in (\ref{eq:u:e26}) yield
(\ref{eq:u:e21}).
This completes the proof.
\end{proof}

\bigskip
\begin{proof}[Proof of Theorem~\ref{thm:HeatEquation}]
First, we assume that $s\not\in\bbZ$.
We have the  following sequence of steps.

\medskip\noindent
(S1)\, We prove that
\begin{align}
  \|\partial_x^{\alpha}u(t,\cdot)
    \|_{L^{p_{|\alpha|}}}
     &\le     \|\partial^{\alpha} u_0\|_{L^{p_{|\alpha|}}},
     \qquad
     |\alpha|\le \lfloor s\rfloor,
     t>0.    \label{eq:u:e3}
\end{align}

Without loss of generality, we assume that
$u$ and $u_0$ are real functions.
Since $\partial_t u = \Delta_x u$,
we have
\begin{equation}\label{eq:u:e2}
  \int_{\bbR^N} (\partial_t u(t,x)) |u(t,x)|^{p_0-2}
     u(t,x) \rmd x
  = \int_{\bbR^N} (\Delta_x u(t,x)) |u(t,x)|^{p_0-2}
     u(t,x) \rmd x .
\end{equation}
From (\ref{eq:u:heat}) we see that for fixed $t>0$
and $0<\varepsilon<t/2$,
all of
$(u(t+\varepsilon,x) - u(t,x))/\varepsilon$,
$ u(t+\varepsilon,x)$
and $u(t,x)$ are
bounded by some function  $F\in L^{p_0}(\bbR^N)$, which is independent of $\varepsilon$.
Applying the inequality
\[
 \Big|  |a|^{p_0} - |b|^{p_0} \Big|
   \le C_{p_0} |a-b|(|a|^{p_0-1} + |b|^{p_0-1}),
 \]
we get
\begin{align*}
  \bigg|\frac{|u(t+\varepsilon,x)|^{p_0} - |u(t,x)|^{p_0}}{\varepsilon}\bigg|
& \le C_{p_0} \frac{|u(t+\varepsilon,x)  - u(t,x)| }{\varepsilon} \cdot (|u(t+\varepsilon,x)|^{p_0-1}
 +|u(t,x)|^{p_0-1}) \\
 &
 \le C_{p_0} |F(x)|^{p_0}.
\end{align*}
Note that $\partial_t |u(t,x)|^{p_0}
= p_0 |u(t,x)|^{p_0-2}u(t,x) \partial_t u(t,x) $.
We see from the Lebesgue  dominated convergence  theorem that
\[
 \int_{\bbR^N} (\partial_t u(t,x)) |u(t,x)|^{p_0-2}
     u(t,x) \rmd x
  =  \frac{1}{p_0} \cdot  \frac{\rmd}{\rmd t} \|u(t,\cdot)\|_{L^{p_0}}^{p_0}.
\]

On the other hand, applying Lemma~\ref{Lm:Lp12}
yields that
\begin{align*}
 \int_{\bbR^N} (\Delta_x u(t,x)) |u(t,x)|^{p_0-2}
     u(t,x) \rmd x
&  = -(p_0-1) \int_{\bbR^N} |\nabla_x u(t,x)|^2 |u(t,x)|^{p_0-2}
   \rmd x.
\end{align*}
Now we see from (\ref{eq:u:e2}) that
\[
  \frac{1}{p_0} \cdot  \frac{\rmd}{\rmd t} \|u(t,\cdot)\|_{L^{p_0}}^{p_0}
  + (p_0-1) \int_{\bbR^N} |\nabla_x u(t,x)|^2 |u(t,x)|^{p_0-2}
   \rmd x=0.
\]
Integrating with respect to $t$, we obtain
\begin{equation}\label{eq:u:e12}
    \frac{1}{p_0}  \|u(t,\cdot)\|_{L^{p_0}}^{p_0}
  + (p_0-1) \int_0^t\int_{\bbR^N} |\nabla_x u(\tau,x)|^2 |u(\tau,x)|^{p_0-2}
   \rmd x \,\rmd \tau=  \frac{1}{p_0}  \|u_0\|_{L^{p_0}}^{p_0}.
\end{equation}
Note that for any multi-index $\alpha$
with $|\alpha|\le \lfloor s\rfloor $,
$\partial_x^{\alpha}u(t,x)$
meets the heat equation with initial data
$\partial^{\alpha}u_0$.
Replacing $\partial_x^{\alpha} u$ for $u$
and $p_{|\alpha|}$ for $p_0$
in
(\ref{eq:u:e12}), respectively,
 we get when $|\alpha|\le \lfloor s\rfloor $,
\begin{align}
&\frac{1}{p_{|\alpha|}}  \|\partial_x^{\alpha}u(t,\cdot)
    \|_{L^{p_{|\alpha|}}}^{p_{|\alpha|}}
  + (p_{|\alpha|}-1) \int_0^t\int_{\bbR^N}
  |\nabla_x \partial_x^{\alpha} u(\tau,x)|^2 |\partial_x^{\alpha}u(\tau,x)|^{p_{|\alpha|}-2}
   \rmd x\,\rmd \tau \nonumber  \\
   &=  \frac{1}{p_{|\alpha|}}  \|\partial^{\alpha} u_0\|_{L^{p_{|\alpha|}}}^{p_{|\alpha|}}.
 \label{eq:u:e3a}
\end{align}
Hence (\ref{eq:u:e3}) is true.

\medskip\noindent
(S2)\, We prove (\ref{eq:u:e0a}).

Set $
g(t,a,x):= \partial_x^{\alpha} u(t,x+ae_j) - \partial_x^{\alpha}u(t,x)$,
where $a\in\bbR$, $|\alpha|= \lfloor s\rfloor$
 and $1\le j\le N$. We have
\[
  \partial_t g(t,a,x) - \Delta_x g(t,a,x)=0.
\]
Replacing $
g(t,a,x) $ for $ u(t,x)$
and $p_{\lceil s\rceil}$ for $p_0$  in (\ref{eq:u:e12}), respectively,
we get
\begin{align}
   &   \frac{1}{p_{\lceil s\rceil}}  \|g(t,a,\cdot)\|_{L^{p_{\lceil s\rceil}}}^{p_{\lceil s\rceil}}
  +  (p_{\lceil s\rceil}-1) \int_0^t
   \int_{\bbR^N}
   |\nabla_x g(\tau,a,x)|^2
   |g(\tau,a,x)|^{p_{\lceil s\rceil}-2}
   \rmd \tau\,\rmd x \nonumber \\
  &=  \frac{1}{p_{\lceil s\rceil}}
   \|g (0,a,\cdot)\|_{L^{p_{\lceil s\rceil}}}^{p_{\lceil s\rceil}},\quad
     |\alpha|= \lfloor s\rfloor.\label{eq:u:e8}
\end{align}
Recall that $\nu_s = s - \lfloor s\rfloor$.
Multiplying  both sides by $1/|a|^{1+\nu_s p_{\lceil s\rceil}}$
  and integrating with
respect to $a\in\bbR$ yields,     for all $t>0$,
\begin{equation}\label{eq:u:e4}
       \Big[\partial_x^{\alpha} u(t,\cdot)\Big]_{W_{\nu_s}^{p_{\lceil s\rceil}}}^{p_{\lceil s\rceil}}
    \le C_{s,\vec p,N}
   \Big[\partial^{\alpha} u_0\Big]_{W_{\nu_s}^{p_{
    \lceil s\rceil}}}^{p_{\lceil s\rceil}},
\end{equation}
where we applied  Proposition \ref{prop:p1}.
It follows from (\ref{eq:u:e3})
and  (\ref{eq:u:e4}) that
\begin{align*}
\sup_{0\le t\le T}
  \|u(t,\cdot)\|_{W_s^{\vec p}}
&= \sup_{0\le t\le T} \bigg(
   \sum_{|\alpha|\le \lfloor s\rfloor}
      \|\partial_x^{\alpha}u(t,\cdot)
    \|_{L^{p_{|\alpha|}}}
   +  \sum_{|\alpha|= \lfloor s\rfloor}  \Big[\partial_x^{\alpha} u(t,\cdot)\Big]_{W_{\nu_s}^{p_{\lceil s\rceil}}}
   \bigg)
   \\
&\le
   \sum_{|\alpha|\le  \lfloor s\rfloor}
      \|\partial^{\alpha}u_0
    \|_{L^{p_{|\alpha|}}}
   +  C_{s,\vec p,N}\sum_{|\alpha|= \lfloor s\rfloor}  \Big[\partial^{\alpha} u_0\Big]_{W_{\nu_s}^{p_{\lceil s\rceil}}}
      \\
 & \le C'_{s,\vec p,N} \|u_0\|_{W_s^{\vec p}},
\end{align*}
which proves (\ref{eq:u:e0a}).

\medskip\noindent
(S3)\,  We prove (\ref{eq:u:e0b}).

Taking derivatives on both sides of (\ref{eq:u:heat}),
we obtain
\[
    \partial_x^{\alpha} u(t,x) =
     \frac{1}{(4\pi t)^{N/2}} \int_{\bbR^N}
        e^{-|x-y|^2/(4t)} \partial^{\alpha}u_0(y) \rmd y, \qquad \forall t>0, x\in\bbR^N, |\alpha|=\lfloor s\rfloor.
\]
It follows from H\"older's inequality that
\begin{equation}\label{eq:u:e30}
  |\partial_x^{\alpha} u(t,x)|
  \le \frac{C}
  {t^{N/(2p_{\lfloor s\rfloor})}}
  \|\partial^{\alpha}u_0\|_{L^{p_{\lfloor s\rfloor}}},
  \quad \forall x\in\bbR^N.
\end{equation}
Since $p_{\lfloor s\rfloor}\le 2$, we have
\[
  |\partial_x^{\alpha} u(t,x)|^{p_{\lfloor s\rfloor}-2}
  \ge \Big(\frac{C}
  {t^{N/(2p_{\lfloor s\rfloor})}}
  \|\partial^{\alpha}u_0\|_{L^{p_{\lfloor s\rfloor}}}\Big)^{p_{\lfloor s\rfloor}-2},
  \quad \forall x\in\bbR^N.
\]
As a consequence of (\ref{eq:u:e3a})  we obtain
\[
  \int_0^T\int_{\bbR^N}
  |\nabla_x \partial_x^{\alpha} u(t,x)|^2
  \Big(\frac{C}
  {t^{N/(2p_{\lfloor s\rfloor})}}
  \|\partial^{\alpha}u_0\|_{L^{p_{\lfloor s\rfloor}}}\Big)^{p_{\lfloor s\rfloor}-2}
   \rmd x\, \rmd t
\le
 \frac{1}{p_{\lfloor s\rfloor}(p_{\lfloor s\rfloor}-1)}
   \|\partial^{\alpha} u_0\|_{L^{p_{
     \lfloor s\rfloor}}}^{p_{\lfloor s\rfloor}}.
\]
Hence,
\begin{equation}\label{eq:u:e33}
      \int_0^T
  t^{(2-p_{\lfloor s\rfloor})N/(2p_{\lfloor s\rfloor} )}
   \|\nabla_x \partial_x^{\alpha}u(t,\cdot)\|_{L^2}^2
   \rmd t
  \le  C'
    \|\partial^{\alpha}u_0\|_{L^{p_{\lfloor s\rfloor}}
      }^2
    ,\quad
     |\alpha|= \lfloor s\rfloor.
\end{equation}
Therefore,
\[
        \int_0^T
  t^{(2-p_{\lfloor s\rfloor})\sigma}
   \|\nabla_x \partial_x^{\alpha}u(t,\cdot)\|_{L^2}^2
   \rmd t    \le  C T^{1-p_{\lfloor s\rfloor}/2}
    \|\partial^{\alpha}u_0\|_{L^{p_{\lfloor s\rfloor}}
      }^2
    ,\quad
     |\alpha|= \lfloor s\rfloor.
\]
Consequently,
\begin{equation}\label{eq:u:e34}
      \int_0^T
  t^{\varrho}
   \|\nabla_x \partial_x^{\alpha}u(t,\cdot)\|_{L^2}^2
   \rmd t    \le  C
   T^{(p_{\lfloor s\rfloor}-p_s)\sigma}
   T^{1-p_{\lfloor s\rfloor}/2}
    \|\partial^{\alpha}u_0\|_{L^{p_{\lfloor s\rfloor}}
      }^2
    ,\quad
     |\alpha|= \lfloor s\rfloor.
\end{equation}

Next we estimate $[\partial_x^{\alpha+\beta}u(t,\cdot)]_{
 W_{\nu_s}^2}$.
 Setting $t=T$ in (\ref{eq:u:e8}), we get
\begin{equation}\label{eq:u:e28}
   \int_0^T
   \int_{\bbR^N}
   |\nabla_x g(t,a,x)|^2
   |g(t,a,x)|^{p_{\lceil s\rceil}-2}
   \rmd t\, \rmd x
   \le   \frac{ 1 }{p_{\lceil s\rceil}(p_{\lceil s\rceil}-1)}
   \|g (0,a,\cdot)\|_{L^{p_{\lceil s\rceil}}}^{p_{\lceil s\rceil}},\quad
     |\alpha|= \lfloor s\rfloor.
\end{equation}
Thus,
\begin{equation}\label{eq:u:e28a}
   \int_0^T
   \int_{\bbR^N}
   \frac{|\nabla_x g(t,a,x)|^2}{|a|^{1+2\nu_s}}
   \cdot
   \frac{|g(t,a,x)|^{p_{\lceil s\rceil}-2}}{|a|^{(p_{\lceil s\rceil}-2)\nu_s}}
   \rmd t\, \rmd x
   \le   C
   \frac{\|g (0,a,\cdot)\|_{L^{p_{\lceil s\rceil}}}^{p_{\lceil s\rceil}}}
   {|a|^{1+p_{\lceil s\rceil}\nu_s}}
   ,\quad
     |\alpha|= \lfloor s\rfloor.
\end{equation}
If $|a|\ge 1$, we see from (\ref{eq:u:e30}) that
\[
  \frac{|g(t,a,x)|}{|a|^{\nu_s}}
  \le
 \frac{C}
  {t^{N/(2p_{\lfloor s\rfloor})}}
  \|\partial^{\alpha}u_0\|_{L^{p_{\lfloor s\rfloor}}}
  \le
 \frac{CT^{1/2}}
  {t^{N/(2p_{\lfloor s\rfloor})+1/2}}
  \|\partial^{\alpha}u_0\|_{L^{p_{\lfloor s\rfloor}}},
  \quad
  \forall 0<t<T, x\in\bbR^N.
\]
If $|a|<1$, we have
\begin{align*}
\frac{|g(t,a,x)|}{|a|^{\nu_s}}
&= \frac{1}{|a|^{\nu_s}}\bigg|
   \int_0^a \partial_{x_j}\partial_x^{\alpha} u(t,x+\tau e_j)
      \rmd \tau
  \bigg|
 \le \|  \partial_{x_j}\partial_x^{\alpha} u(t,\cdot)\|_{L^{\infty}}.
\end{align*}
By (\ref{eq:u:heat}),
\[
 \partial_{x_j}   \partial_x^{\alpha} u(t,x) =
     \frac{1}{(4\pi t)^{N/2}} \int_{\bbR^N}
       \frac{-(x_j-y)}{ 2t^{1/2}} e^{-|x-y|^2/(4t)} \partial^{\alpha}u_0(y) \rmd y .
\]
Therefore, for $|a|<1$,
\[
 \frac{|g(t,a,x)|}{|a|^{\nu_s}}
 \le  \frac{C}
  {t^{N/(2p_{\lfloor s\rfloor})+1/2}}
  \|\partial^{\alpha}u_0\|_{L^{p_{\lfloor s\rfloor}}},
 \quad  \forall t>0, x\in\bbR^N.
\]
Consequently
\[
  \frac{|g(t,a,x)|^{p_{\lceil s\rceil}-2}}{|a|^{(p_{\lceil s\rceil}-2)\nu_s}}
  \ge \bigg(\frac{C\max\{1,T^{1/2}\}}
  {t^{N/(2p_{\lfloor s\rfloor})+1/2}}
  \|\partial^{\alpha}u_0\|_{L^{p_{\lfloor s\rfloor}}}
  \bigg)^{p_{\lceil s\rceil}-2},\quad
  \forall 0<t<T, \, a\in\bbR,\, x\in\bbR^N.
\]
It follows from (\ref{eq:u:e28a}) that
\begin{align*}
&\hskip -10mm\int_0^T
   \int_{\bbR^N}
   t^{(2-p_{\lceil s\rceil})(N/(2p_{\lfloor s\rfloor})+1/2)}
   \frac{|\nabla_x g(t,a,x)|^2}{|a|^{1+2\nu_s}}
       \rmd t\,\rmd x \\
   &\le   \left(C\max\{1,T^{1/2}\}
   \|\partial^{\alpha}u_0\|_{L^{p_{\lfloor s\rfloor}}
      }\right)^{2-p_{\lceil s\rceil} }
   \frac{\|g (0,a,\cdot)\|_{L^{p_{\lceil s\rceil}}}^{p_{\lceil s\rceil}}}
   {|a|^{1+p_{\lceil s\rceil}\nu_s}}
   ,\quad
     |\alpha|= \lfloor s\rfloor.
\end{align*}
Integrating with respect to $a\in\bbR$ yields
\begin{align}
&  \hskip -10mm \int_0^T
   t^{(2-p_{\lceil s\rceil})(N/(2p_{\lfloor s\rfloor})+1/2)}
  \Big[\nabla_x \partial_x^{\alpha}u(t,\cdot )\Big]_{W_{\nu_s}^2}^2
       \rmd t \nonumber \\
  & \le   \left(C\max\{1,T^{1/2}\}
   \|\partial^{\alpha}u_0\|_{L^{p_{\lfloor s\rfloor}}
      }\right)^{2-p_{\lceil s\rceil} }
  \Big[\partial^{\alpha}u_0\Big]_{W_{\nu_s}^{p_{\lceil s\rceil}}}^{p_{\lceil s\rceil}}
   \nonumber \\
 &\le   C'\max\{1,T^{1-p_{\lceil s\rceil}/2}\}
  \|u_0\|_{W_s^{\vec p}}^2    ,
   \quad
     |\alpha|= \lfloor s\rfloor.\label{eq:u:e32}
\end{align}
Thus, we obtain
\begin{align}
   \int_0^T
   t^{\varrho }
  \Big[\nabla_x \partial_x^{\alpha}u(t,\cdot )\Big]_{W_{\nu_s}^2}^2
       \rmd t
 &\le   C'
    T^{(p_{\lceil s\rceil}-p_s)\sigma}
    \max\{1,T^{1-p_{\lceil s\rceil}/2}\}
  \|u_0\|_{W_s^{\vec p}}^2    ,
   \quad
     |\alpha|= \lfloor s\rfloor.\label{eq:u:e32a}
\end{align}
On the other hand, applying (\ref{eq:u:e3})
yields that
\begin{equation}\label{eq:u:e34a}
      \int_0^T
  t^{\varrho}
   \| \partial_x^{\alpha}u(t,\cdot)\|_{L^{p_{|\alpha|}}}^2
   \rmd t    \le
   T^{1+\varrho }
   \| \partial^{\alpha}u_0\|_{L^{p_{|\alpha|}}}^2
    ,\quad
    \mbox{ when }
     |\alpha|\le  \lfloor s\rfloor.
\end{equation}
Summing up
(\ref{eq:u:e34}), (\ref{eq:u:e32a}) and (\ref{eq:u:e34a}), we get
\begin{align*}
      \int_0^T
  t^{\varrho}
   \|  u(t,\cdot)\|_{W_{s+1}^{(\vec r,2)}}^2
   \rmd t    \le
      C(1+T^{\max\{1+\varrho,
       1-p_{\lceil s\rceil} /2+
      |p_{\lceil s\rceil} - p_{\lfloor s\rfloor}|
      \sigma
      \}})\|  u_0\|_{W_s^{\vec p}}^2.
\end{align*}
Hence (\ref{eq:u:e0b}) is valid.

\medskip\noindent
(S4)\,
  We prove (\ref{eq:u:e0c}).

Set $
h(t,a,x):=\partial_x^{\alpha+\beta}  u(t,x+ae_j) - \partial_x^{\alpha+\beta} u(t,x)$, where
$|\alpha|=\lfloor s\rfloor$,
$|\beta|=1$
and $1\le j\le N$.
We have
\[
  \partial_t h(t,a,x) - \Delta_x h(t,a,x)=0.
\]
Taking the $L_x^2$ inner product with
$t^{1+\varrho} h(t,a,x)$, we get
\[
  \int_{\bbR^N}
  (\partial_t h(t,a,x))
  t^{1+\varrho} h (t,a,x)\rmd x
    - \int_{\bbR^N}
    (\Delta_x h(t,a,x))t^{1+\varrho} h(t,a,x) \rmd x
    =0.
\]
Hence
\[
  \frac{1}{2}\cdot \frac{\rmd}{\rmd t}
     (t^{1+\varrho} \|h(t,a,\cdot)\|_{L^2}^2)
     + t^{1+\varrho}     \|\nabla_x h(t,a,\cdot)\|_{L^2}^2
           =\frac{1+\varrho}{2} t^{\varrho}  \|h(t,a,\cdot)\|_{L^2}^2.
\]
Integrating with respect to $t\in [0,T]$ yields
\begin{align}
\frac{T^{1+\varrho}}{2} \|h(T,a,\cdot)\|_{L^2}^2
+ \int_0^T    t^{1+\varrho} \|\nabla_x h(t,a,\cdot)\|_{L^2}^2
         \rmd t
&= \frac{1+\varrho}{2} \int_0^T
  t^{\varrho}\|h(t,a,\cdot)\|_{L^2}^2\rmd t
.  \label{eq:u:e15}
\end{align}
Consequently,
\begin{align}
\frac{T^{1+\varrho}}{2} \Big[\partial_x^{\alpha+\beta} u(T,\cdot)\Big]_{W_{\nu_s}^2}^2
&\!+\!  \int_0^T t^{1+\varrho}
 \Big[\nabla_x \partial_x^{\alpha+\beta} u(t,\cdot)\Big]_{W_{\nu_s}^2}^2
   \rmd t
 = \frac{1+\varrho}{2}\! \int_0^T
   t^{1+\varrho}\Big[\partial_x^{\alpha+\beta}  u(t,\cdot)\Big]_{W_{\nu_s}^2}^2\rmd t\nonumber \\
&\le C_{s,\vec p,N} T^{1+(p_{\lceil s\rceil}-p_s)\sigma}\max\{1,T^{1-p_{\lceil s\rceil}/2}\}
\|u_0\|_{W_{s}^{\vec p}}^2
, \nonumber
\end{align}
applying (\ref{eq:u:e32a}).
Hence  for $|\alpha|=\lfloor s\rfloor$ and
$|\beta|=1$,
\begin{equation}
\label{eq:u:e9}
\int_0^T t^{1+\varrho} \Big[\nabla_x \partial_x^{\alpha+\beta} u(t,\cdot)\Big]_{W_{\nu_s}^2}^2
   \rmd t
\le    C_{s,\vec p,N} T^{1+(p_{\lceil s\rceil}-p_s)\sigma}\max\{1,T^{1-p_{\lceil s\rceil}/2}\}
\|u_0\|_{W_{s}^{\vec p}}^2.
\end{equation}
On the other hand,
substituting
$ \partial_x^{\alpha+\beta}u(t,x)$
for
$h(t,a,x)$
in (\ref{eq:u:e15}), where $|\beta|=1$
and $|\alpha|= \lfloor s\rfloor$,
we get
\begin{align}
\frac{T^{1+\varrho}}{2} \| \partial_x^{\alpha+\beta}u(T, \cdot)\|_{L^2}^2
+ \int_0^T t^{1+\varrho} \|\nabla_x \partial_x^{\alpha+\beta}u(t, \cdot)\|_{L^2}^2
   \rmd t
&= \frac{1+\varrho}{2} \int_0^T
  t^{\varrho}\| \partial_x^{\alpha+\beta}u(t ,\cdot)\|_{L^2}^2\rmd t \nonumber \\
&\le
    C
   T^{(p_{\lfloor s\rfloor}-p_s)\sigma}
   T^{1-p_{\lfloor s\rfloor}/2}
    \|\partial^{\alpha}u_0\|_{L^{p_{\lfloor s\rfloor}}
      }^2 , \nonumber
\end{align}
applying (\ref{eq:u:e34}).
Hence
\begin{align}
 \int_0^T t^{1+\varrho} \|\nabla_x \partial_x^{\alpha+\beta}u(t, \cdot)\|_{L^2}^2
   \rmd t
&\le  C
   T^{(p_{\lfloor s\rfloor}-p_s)\sigma}
   T^{1-p_{\lfloor s\rfloor}/2}
    \|\partial^{\alpha}u_0\|_{L^{p_{\lfloor s\rfloor}}
      }^2
. \label{eq:u:e9a}
\end{align}
Moreover,
we see from (\ref{eq:u:e34}) that
when $|\alpha|=\lfloor s\rfloor$
and $|\beta|=1$,
\begin{equation}\label{eq:u:e34c}
      \int_0^T
  t^{1+\varrho}
   \|  \partial_x^{\alpha+\beta}u(t,\cdot)\|_{L^2}^2
   \rmd t    \le  C
   T^{1+(p_{\lfloor s\rfloor}-p_s)\sigma}
   T^{1-p_{\lfloor s\rfloor}/2}
    \|\partial^{\alpha}u_0\|_{L^{p_{\lfloor s\rfloor}}
      }^2
     .
\end{equation}
When $|\alpha|\le \lfloor s\rfloor$, we
 apply  (\ref{eq:u:e3}) to obtain
\begin{equation}\label{eq:u:e34d}
      \int_0^T
  t^{1+\varrho}
   \| \partial_x^{\alpha}u(t,\cdot)\|_{L^{p_{|\alpha|}}}^2
   \rmd t    \le
   T^{2+\varrho }
   \| \partial^{\alpha}u_0\|_{L^{p_{|\alpha|}}}^2
     .
\end{equation}
Combining  (\ref{eq:u:e9}), (\ref{eq:u:e9a}),
 (\ref{eq:u:e34c}) and (\ref{eq:u:e34d})
we deduce
\[
  \int_0^T
 t^{1+\varrho} \|   u(t,\cdot)\|_{W_{s+2}^{(\vec r,2,2) }}^2
     \rmd t
   \le
     C_{s,\vec p, N} \left(
     T^{(p_{\lfloor s\rfloor}-p_s)\sigma
       +1-p_{\lfloor s\rfloor}/2}
   +T^{2+\varrho}\right)
     \|u_0\|_{W_s^{\vec p}}^2.
\]
Hence (\ref{eq:u:e0c}) is true.

This completes the proof of (\ref{eq:u:e0a}), (\ref{eq:u:e0b})
and (\ref{eq:u:e0c})
in the case $s\not\in\bbZ$.

When $s\in\bbZ$, we apply similar arguments; we only provide a sketch.
First, (\ref{eq:u:e0a}) follows from (\ref{eq:u:e3}).
Then we get (\ref{eq:u:e0b})
by (\ref{eq:u:e3}) and (\ref{eq:u:e33}).
Finally, (\ref{eq:u:e0c}) follows from
(\ref{eq:u:e0b}) and
(\ref{eq:u:e9a}).
In both cases,  (\ref{eq:u:e0a}),
(\ref{eq:u:e0b}) and (\ref{eq:u:e0c}) are valid.
It follows that
for any $0<q<2/(2+\varrho)$,
\begin{align*}
\int_0^T
  \|   u(t,\cdot)\|_{W_{s+2}^{(\vec r,2,2) }}^q
     \rmd t
 &=   \int_0^T
   t^{-q(1+\varrho)/2} t^{q(1+\varrho)/2}
  \|   u(t,\cdot)\|_{W_{s+2}^{(\vec r,2,2) }}^q
     \rmd t    \\
 &\le \Big(\int_0^T t^{-q(1+\varrho)/(2-q)}\rmd t
    \Big)^{1-q/2}
    \Big(
    \int_0^T t^{(1+\varrho)} \|u(t,\cdot)\|_{W_{s+2}^{(\vec r,2,2) }}^2
     \rmd t
    \Big)^{q/2}    \\
 &\le C_{s,\vec p, N} T^{1-(1+\varrho/2)q}  (1+T^{\theta_2})^{q/2}
    \|u_0\|_{W_s^{\vec p}}^q.
\end{align*}
This completes the proof.
\end{proof}

\begin{Remark}\upshape
\mbox{}
\begin{enumerate}
\item
Whenever $p_{\lfloor s\rfloor}
=p_{\lceil s\rceil} = 2$,
we have $\varrho = 0$. In this case,
the local estimate
(\ref{eq:u:q}) is valid for all
$0<q<1$, which coincides with
\cite[Lemma 2.1]{FeffermanMcCormickRobinsonRodrigo2017}.

Moreover, (\ref{eq:u:e0b}),
(\ref{eq:u:e0c}) and (\ref{eq:u:q}) now turn out to be
\begin{align*}
 &\int_0^T \|u(t,\cdot)\|_{W_{s+1}^{(\vec p, 2)}}^2 \rmd t
 \le C_{s,\vec p,N} (1+T)\|u_0\|_{W_s^{\vec p}}^2,
 \\
   &\int_0^T t\|u(t,\cdot)\|_{W_{s+2}^{(\vec p, 2, 2)}}^2 \rmd t
  \le C_{s,\vec p,N} (1+T^2)\|u_0\|_{W_s^{\vec p}}^2,
 \\
&   \int_0^T \|u(t,\cdot)\|_{W_{s+2}^{(\vec p, 2,2)}}^q
  \rmd t
  \le C_{s,\vec p,N} T^{1-q}(1+T^2)^{q/2}
  \|u_0\|_{W_s^{\vec p}}^q,\quad 0<q<1.
\end{align*}

\item
Theorem~\ref{thm:HeatEquation}
extends the local estimates for initial data in
 classical Sobolev spaces $H^s$.
For example,
consider the initial data
\[
  u_0(x) = \frac{1}{(1+|x|^2)^{\delta/2}}.
\]
When $\max\{0, N/2-1\} < \delta<N/2$
and $1<s<2$,
$u_0\in W_s^{(p_0,2,2)}$ for all $p_0 >  N/\delta  $.
Now  Theorem~\ref{thm:HeatEquation}
gives local energy estimates for
the heat equation with initial data $u_0$,
while
one has no local estimates
with classical Sobolev spaces since $u_0\not\in L^2$.
\end{enumerate}
\end{Remark}

\subsection{Convergence of Schr\"odinger Operators}

In this  subsection, we study the convergence of Schr\"odinger operators.

Take some function $\varphi$ such that
$\hat \varphi \in C_c^{\infty}$, $0\le \hat\varphi(\omega)\le 1$,
$\hat \varphi(\omega)=1$ for $|\omega|<1$
and $\hat \varphi(\omega)=0$ for $|\omega|>2$.
Set $\hat f_1 = \hat \varphi\cdot \hat f$ and $f_2 = f - f_1$.

Note that $(f*\varphi)\,\hat{} =  \hat \varphi\hat f$ (for a proof, see
\cite[Theorem 7.19]{Rudin1991}). Hence
$f_1 = f*\varphi$ and $f_2 = f - f*\varphi$.
Now we rewrite $e^{it(-\Delta)^{a/2}}$ as
\begin{equation}\label{eq:s:phi}
  e^{it(-\Delta)^{a/2}}f = e^{it(-\Delta)^{a/2}}f_1  + e^{it(-\Delta)^{a/2}}f_2.
\end{equation}
If $f\in H^s$, then
$ \hat f  $ is locally integrable. Hence  the convergence for
\[
 \lim_{t\rightarrow 0}
   e^{it(-\Delta)^{a/2}}f_1(x)
 = \lim_{t\rightarrow 0}  \frac{1}{(2\pi)^N} \int_{\bbR^N} e^{ i(x\cdot\omega + t |\omega|^a )}
     \hat f_1(\omega) \rmd \omega
\]
is obvious.

However, for $f\in W_s^{\vec p}$ with $p_0>2$,
we do not know whether $\hat f$ is locally integrable.
So we have to deal with the term $e^{it(-\Delta)^{a/2}}f_1$
with new method.

For the case $p_{\lceil s\rceil} =  2$, we show that
if $e^{it(-\Delta)^{a/2}}f$ is convergent as $t$ tends to zero
for all functions in $H^s$ for some $s>0$,
then the same is true for all
functions in $W_s^{\vec p}$ with the same index $s$.

\begin{Theorem}\label{thm:convergence:2}
Let $s>0$, $a>1$, $\vec p = (p_0,\ldots, p_{\lceil s\rceil})$
with $p_{\lceil s\rceil} =  2$
and $1\le p_l<\infty$ for all $0\le l\le \lceil s\rceil-1$.
Suppose that
for all functions $f\in H^s(\bbR^N)$,
\begin{equation}\label{eq:Schr:1:a}
\lim_{t\rightarrow 0}e^{it(-\Delta)^{a/2}}f(x) = f(x),\quad a.e.
\end{equation}
We have:
\begin{enumerate}
\item If $0<s<N/2$,
then (\ref{eq:Schr:1:a}) is valid for all functions
$f\in W_s^{\vec p}$.

\item If $s\ge N/2$ and $a=2$, then
(\ref{eq:Schr:1:a}) is also true if we  interpret
the operator
$e^{-it\Delta}$ as
\begin{equation}\label{eq:Delta}
  e^{-it\Delta} f(x)=\lim_{\varepsilon\rightarrow 0}
      \frac{1}{(2\pi)^N} \int_{\bbR^N} e^{ i(x\cdot\omega + t |\omega|^2 )}
      \Big( 1- \varphi_0(\frac{\omega}{\varepsilon})\Big)\hat f(\omega) \rmd \omega,
\end{equation}
where $\varphi_0\in C_c^{\infty}(\bbR)$
satisfies that $\varphi_0(x)=1$ for $|x|<1$
and $\varphi_0(x)=0$ for $|x|>2$,
the limit exists almost everywhere,
and the limit is independent of  the choice  of $\varphi_0$ with the aforementioned properties.
\end{enumerate}
\end{Theorem}

For the case $1< p_{\lceil s\rceil} <  2$
and $a=2$, we have similar results.

\begin{Theorem}\label{thm:convergence:p}
Let $s>0$,  $\vec p = (p_0,\ldots, p_{\lceil s\rceil})$
with $1<p_{\lceil s\rceil}< 2$
and $1\le p_l<\infty$ for all $0\le l\le \lceil s\rceil-1$.
Let
\begin{equation}\label{eq:beta s}
  \beta_s = \begin{cases}
    s,  & s\in\bbZ, \\
\lfloor s\rfloor + \nu_s(p_{\lceil s\rceil}-1), & s\not\in\bbZ.
  \end{cases}
\end{equation}
We have:
\begin{enumerate}
\item If
\begin{equation}\label{eq:Schr:5}
  \frac{1}{p_{\lceil s\rceil}} -\frac{N}{2(N+1)}< \frac{\beta_s}{N}
   < \frac{1}{p_{\lceil s\rceil}},
\end{equation}
then   for any
$f\in W_s^{\vec p}$,
\begin{equation}\label{eq:Schr:1}
\lim_{t\rightarrow 0}e^{-it\Delta}f(x) = f(x),\quad a.e.
\end{equation}

\item If $\beta_s  \ge N/p_{\lceil s\rceil}$, then
(\ref{eq:Schr:1}) is also true if we interpret
the operator
$e^{-it\Delta}$ as in (\ref{eq:Delta}).
\end{enumerate}
\end{Theorem}

\begin{Remark}
If  $0<s<1$,  $p_0> N/\delta  >  p_1 =2 $
and $2(\delta+s)>N$,
then the function $f$ defined in Example~\ref{Ex:ex1}
satisfies that  $f\in W_s^{(p_0,2)} \setminus W_s^{(2,2)}$.
That is,
$W_s^{(p_0,2)}$ is not contained in $W_s^{(2,2)} = H^s$.
\end{Remark}

For the case $p_0>2$, the Fourier transform of  functions in
$W_s^{\vec p}$ might be  distributions.
In this case, we show that $\hat f$ is
the distributional limit of a sequence of locally
integrable functions. Moreover, we prove that
$\hat f$ is locally integrable if $ \beta_s /N< 1/p_{\lceil s\rceil}$.

\begin{Lemma}\label{Lm:frac:Sobolev}
Suppose that $s>0$ and $\vec p = (p_0,\ldots, p_{\lceil s\rceil})$ with
$1<p_{\lceil s\rceil}\le 2$
and $1\le p_l<\infty$ for $0\le l\le \lceil s\rceil-1$.
Let $\beta_s$ be defined by (\ref{eq:beta s}).
For any $f\in W_s^{\vec p}$,
$\hat f$ coincides with a function $\lambda$ in the domain
$\bbR^N\setminus\{0\}$
such that $|x|^{\beta_s} \lambda(x) \in L^{p'_{\lceil s\rceil}}$
and
$\hat f = \lim_{\varepsilon\rightarrow 0}
  (1 - \varphi_0(x/\varepsilon)) \lambda(x)$
in the distributional sense,
where
$\varphi_0 \in C^{\infty}(\bbR^N)$ satisfies that
$\varphi_0(x) =1$ for $|x|<1$ and $\varphi_0(x)=0$ for $|x|>2$.

Consequently, if $\beta_sp_{\lceil s\rceil}<N$,
then $\lambda$ is locally integrable and $\hat f  = \lambda$.
\end{Lemma}

\begin{proof}
First, we show that there is some function $\lambda$ such that $|x|^{\beta_s} \lambda(x)\in L^{p'_{\lceil s\rceil}}$
and
\begin{equation}\label{eq:s:7}
  \langle \hat f, \varphi \rangle  = \int_{\bbR^N} \lambda (x)  \varphi(x)\rmd x,
  \qquad \forall \varphi\in C_c^{\infty}(\bbR^N\setminus\{0\}).
\end{equation}

There are three cases:

\medskip\noindent
(A1)\, $0<s<1$.
\medskip

In this case, $\vec p= (p_0, p_1)$.
We see from Proposition~\ref{prop:p1} that
\[
  \iint_{\bbR^N\times\bbR^N} \frac{|f(x)-f(y)|^{p_1}}{|x-y|^{N+p_1s}} \rmd x\,\rmd y<\infty
\]
is equivalent to
\begin{equation}\label{eq:s:4}
\iint_{\bbR^N\times\bbR} \frac{|f(x)-f(x+ae_j)|^{p_1}}{|a|^{1+p_1s}} \rmd x\, \rmd a<\infty,
\qquad \forall 1\le  j\le N.
\end{equation}
Applying Fubini's theorem, we get that for almost all real numbers $a$,
$f - f(\cdot + ae_j)\in L^{p_1}$.
It is easy to see that $f(\cdot -ae_j)\,\hat{} = \phi_j(a\cdot ) \hat f$ in the distributional
sense, where
\[
  \phi_j(x) = e^{- i x_j} .
\]
Hence
\[
  (f - f(\cdot + ae_j))\, \hat{} = (1 - \phi_j(a\cdot )) \hat f.
\]
Since $f - f(\cdot + ae_j)\in L^{p_1}$ for almost all $a>0$,
by the Hausdorff-Young  inequality,
there exist functions $g_a\in L^{p'_1}$
such that
\[
  (f - f(\cdot + ae_j))\, \hat{}  =    g_a.
\]
Hence for  any $\varphi\in C_c^{\infty} (B(0,1/a)\setminus\{0\})$,
\begin{align*}
\langle (1 - \phi_j(a\cdot )) \hat f , \varphi\rangle
  &=   \int_{\bbR^N}  g_a(x) \varphi(x)\rmd x.
\end{align*}
Therefore,
\begin{align*}
\langle |1 - \phi_j(a\cdot )|^2 \hat f , \varphi\rangle
&= \langle (1 - \phi_j(a\cdot )) \hat f , (1 - \phi_j(a\cdot ))^*\varphi\rangle
  &=   \int_{\bbR^N} (1 - \phi_j(ax ))^* g_a(x) \varphi(x)\rmd x,
\end{align*}
where $z^*$ denotes the conjugate of a complex number $z$.
It follows that
\begin{align*}
\Big\langle \sum_{j=1}^N
  |1 - \phi_j(a\cdot )|^2 \hat f , \varphi\Big\rangle
  &=   \int_{\bbR^N} \sum_{j=1}^N(1 - \phi_j(ax ))^* g_a(x) \varphi(x)\rmd x,
\end{align*}
Set
\begin{equation}\label{eq:s:3}
\lambda_a(x) = \frac{\sum_{j=1}^N(1 - \phi_j(ax ))^* g_a(x)}
   {\sum_{j=1}^N
  |1 - \phi_j(ax )|^2 }.
\end{equation}
Note that $\sum_{j=1}^N
  |1 - \phi_j(a\cdot )|^2 $ has no zero in the area $0<|x|<1/a$.
Hence
\[
\frac{\varphi }{  \sum_{j=1}^N
  |1 - \phi_j(a\cdot )|^2} \in C_c^{\infty}(\bbR^N\setminus\{0\}) .
  \]
It follows that
\[
  \Big \langle\sum_{j=1}^N
  |1 - \phi_j(a\cdot )|^2 \hat f , \frac{\varphi}{\sum_{j=1}^N
  |1 - \phi_j(a\cdot )|^2}\Big \rangle
  = \int_{\bbR^N} \lambda_a(x)  \varphi(x)\rmd x.
\]
On the other hand, 
\begin{align*}
\Big \langle\sum_{j=1}^N
  |1 - \phi_j(a\cdot )|^2 \hat f , \frac{\varphi}{\sum_{j=1}^N
  |1 - \phi_j(a\cdot )|^2}\Big \rangle
&= \langle \hat f, \varphi \rangle  .
\end{align*}
Hence
\begin{equation}\label{eq:s:6}
  \langle \hat f, \varphi \rangle  = \int_{\bbR^N} \lambda_a(x)  \varphi(x)\rmd x,
  \qquad \forall \varphi\in C_c^{\infty} (B(0,1/a)\setminus\{0\}).
\end{equation}
That is, $\hat f$ coincides with a function $\lambda_a$ in the domain $B(0,1/a)\setminus\{0\}$.
By the uniqueness, we have
\[
  \lambda_{a}(x) = \lambda_{a'}(x),\qquad a'<a, 0<|x|<\frac{1}{a}.
\]
Take a sequence $\{a_n:\, n\ge 1\}$ such that
$a_n\rightarrow 0$. Then the limit
\[
  \lambda(x) := \lim_{n\rightarrow \infty} \lambda_{a_n}(x)
\]
exists for almost all $x\ne 0$.
Now we see from (\ref{eq:s:6}) that
 (\ref{eq:s:7}) is true.

Since $f-f(\cdot +ae_j)\in L^{p_1}$, we see from the above equation that for all $a\in\bbR$,
\[
  \Big(f-f(\cdot +ae_j)\Big)\,\hat{}\, (x)
  = \lambda (x) (1 - \phi_j(ax)),\qquad \forall \varepsilon<|x|<A.
\]
Applying the Hausdorff Young  inequality, we get
\begin{align}
&\iint_{\bbR^N\times\bbR} \frac{|f(x)-f(x+ae_j)|^{p_1}}{|a|^{1+p_1s}} \rmd x\, \rmd a
    \nonumber \\
&\ge
  \iint_{\bbR^N\times\bbR} \frac{|(f-f(\cdot +ae_j))\, \hat{}(x)|^{p'_1}}
    {|a|^{1+p_1s}} \rmd x\, \rmd a
    \nonumber \\
&= \lim_{\substack{\varepsilon\rightarrow 0 \\
A\rightarrow \infty}}\int_{\varepsilon<|x|<A}  \int_{\bbR}\frac{|(f-f(\cdot +ae_j))\,
\hat{}(x)|^{p'_1}}{|a|^{1+p_1s}} \rmd a  \,  \rmd x
    \nonumber \\
&=\lim_{\substack{\varepsilon\rightarrow 0 \\
A\rightarrow \infty}} \int_{\varepsilon<|x|<A}   \int_{\bbR}
           \frac{|\lambda (x) (1 - \phi_j(ax))|^{p'_1}}
             {|a|^{1+p_1s}}  \rmd a  \,  \rmd x  \nonumber   \\
&= \int_{\bbR^N} \int_{\bbR}
    \frac{|\lambda (x)|^{p'_1}
       |2\sin ( ax_j/2)|^{p'_1}}{|a|^{1+p_1s}}  \rmd a \,  \rmd x   \nonumber \\
&= \int_{\bbR^N}   |x_j|^{p_1s} |\lambda (x)|^{p'_1}
    \int_{\bbR} \frac{| 2\sin (a/2)|^{p'_1} }{|a|^{1+p_1s }}
\rmd a  \,  \rmd x \nonumber  \\
&= C_{s,j} \int_{\bbR^N} |x_j|^{p_1s} |\lambda (x)|^{p'_1} \rmd x  . \label{eq:Schr:2}
\end{align}
Hence
\begin{equation}\label{eq:s:e16}
  \int_{\bbR^N}   |x_j|^{p_1s} |\lambda (x)|^{p'_1} \rmd x
  \le \frac{1}{C_{s,j}}
    \iint_{\bbR^N\times\bbR} \frac{|f(x)-f(x+ae_j)|^{p_1}}{|a|^{1+p_1s}} \rmd x\, \rmd a
  <\infty.
\end{equation}That is, $  |x_j|^{s(p_1-1)} \lambda(x) \in L^{p'_1}$ for all $1\le j\le N$.
Consequently,
$|x|^{s(p_1-1)} \lambda (x)\in L^{p'_1}$.

\medskip\noindent
(A2)\, $s=k\ge 1$ is an integer.
\medskip

For all multi-index $\alpha$ with
$|\alpha| = k$, let $h_{\alpha}$ be the Fourier transform
of $D^{\alpha}f$
and $\psi_{\alpha}(x) = (  i x)^{\alpha}$.
Then $h_{\alpha} = \psi_{\alpha}\hat f $ is a function in $L^{p'_k}$.
Consequently,
\[
| \omega|^s |\hat f(\omega)| \approx
\sum_{|\alpha|=k} |h_{\alpha}(\omega)| \in L^{p'_k}  .
\]

For any $\varphi \in C_c^{\infty}(\bbR^N)$ with $\varphi(x)=0$ in a neighbourhood
of $0$, since $h_{\alpha} = \psi_{\alpha}\hat f $ is a function in $L^{p'_k}$, we have
\begin{align*}
  \langle \psi_{\alpha} \hat f, \varphi  \rangle
= \int_{\bbR^N} h_{\alpha}(x)  \varphi(x) \rmd x .
\end{align*}
Hence
\begin{align*}
  \langle \psi_{\alpha}^2 \hat f, \varphi  \rangle
  =\langle \psi_{\alpha}  \hat f, \psi_{\alpha} \varphi  \rangle
= \int_{\bbR^N} h_{\alpha}(x) \psi_{\alpha}(x) \varphi(x) \rmd x .
\end{align*}
Therefore,
\begin{align*}
  \Big\langle \sum_{|\alpha|=k} \psi_{\alpha}^2 \hat f, \varphi \Big \rangle
= \int_{\bbR^N} \sum_{|\alpha|=k}  h_{\alpha}(x) \psi_{\alpha}(x) \varphi(x) \rmd x .
\end{align*}
Note that $ \sum_{|\alpha|=k} \psi_{\alpha}^2(x) $  has no zero other than
$x=0$.
Substituting $\varphi/ \sum_{|\alpha|=k} \psi_{\alpha}^2 $ for  $\varphi$ in the above equation,
we get
\begin{align*}
\langle \hat f, \varphi  \rangle
 = \Big\langle \sum_{|\alpha|=k} \psi_{\alpha}^2 \hat f, \frac{\varphi}{\sum_{|\alpha|=k} \psi_{\alpha}^2} \Big \rangle
= \int_{\bbR^N}  \frac{ \sum_{|\alpha|=k} h_{\alpha}(x) \psi_{\alpha}(x)}
   {\sum_{|\alpha|=k} \psi_{\alpha}(x)^2} \varphi(x) \rmd x .
\end{align*}
Let
\[
  \lambda(x) =  \frac{ \sum_{|\alpha|=k} h_{\alpha}(x) \psi_{\alpha}(x)}
   {\sum_{|\alpha|=k} \psi_{\alpha}(x)^2}.
\]
We get
 (\ref{eq:s:7})  and
 $|x|^s |\lambda (x)|
 \lesssim \sum_{\alpha} |h_{\alpha}(x)| \in L^{p'_{\lceil s\rceil}}$.

\medskip\noindent
(A3)\, $s>1$ is not an integer.
\medskip

For any multi-index $\alpha$ with $|\alpha|=\lfloor s\rfloor$,
set $\psi_{\alpha}(x) = ( i x)^{\alpha}$.
Substituting $(p_{\lceil s\rceil-1},p_{\lceil s\rceil})$, $\nu_s$  and $D^{\alpha} f$
for $(p_0,p_1)$, $s$ and $f$ respectively in (A1), we get
$|\omega|^{\nu_s(p_{\lceil s\rceil}-1)} (D^{\alpha} f)\,\hat{}(\omega) \in L^{p'_{\lceil s\rceil}}$
 and there exists a function $\lambda_{\alpha}$ such that
 $|\omega|^{\nu_s(p_{\lceil s\rceil}-1)} \lambda_{\alpha}(\omega) \in L^{p'_{\lceil s\rceil}}$
 and
for any   $\varphi \in C_c^{\infty}(\bbR^N)$ with $\varphi(x)=0$ in a neighbourhood
of $0$,
\begin{align*}
\langle \psi_{\alpha}\hat f, \varphi \rangle
&=\langle (D^{\alpha} f)\,\hat{}, \varphi \rangle
= \int_{\bbR^N} \lambda_{\alpha} (x)\varphi(x)  \rmd x .
\end{align*}
With similar arguments as in the previous case, we get
\begin{align*}
\langle \hat f, \varphi  \rangle
= \int_{\bbR^N}   \frac{ \sum_{|\alpha|=\lfloor s\rfloor}\lambda_{\alpha}(x) \psi_{\alpha}(x)}
   {\sum_{|\alpha|=\lfloor s\rfloor} \psi_{\alpha}(x)^2} \varphi(x) \rmd x .
\end{align*}
Let
\[
  \lambda(x) =  \frac{ \sum_{|\alpha|=\lfloor s\rfloor} \lambda_{\alpha}(x) \psi_{\alpha}(x)}
   {\sum_{|\alpha|=\lfloor s\rfloor} \psi_{\alpha}(x)^2}.
\]
We get
 (\ref{eq:s:7})  and
 $|x|^{\lfloor s\rfloor + \nu_s(p_{\lceil s\rceil}-1)} \lambda (x)|
 \lesssim \sum_{\alpha} |x|^{ \nu_s(p_{\lceil s\rceil}-1)}
   |\lambda_{\alpha}(x)| \in L^{p'_{\lceil s\rceil}}$.

Recall that $C_c^{\infty}$ is dense in $\mathscr S$.
For any $\varphi\in\mathscr S$ with  $\varphi(x)=0$ in a neighbourhood of $0$,
 we
see from (\ref{eq:s:7}) that
\begin{equation}\label{eq:s:e3}
  \langle \hat f,  \varphi\rangle
    = \int_{\bbR^N} \lambda(x) \varphi(x)\rmd x.
\end{equation}

Take some $\varphi_0 \in C^{\infty}(\bbR^N)$ such that
$\varphi_0(x) =1$ for $|x|<1$ and $\varphi_0(x)=0$ for $|x|>2$.
We conclude that
\[
  \lim_{\varepsilon\rightarrow 0} \varphi_0\big(\frac{\cdot}{\varepsilon}\big) \hat f
   =0,\qquad \mbox{in} \quad \mathscr S'.
\]
It suffices to show that for any $\varphi \in \mathscr S$,
\begin{equation}\label{eq:s:2}
  \lim_{\varepsilon\rightarrow 0}
  \langle  \varphi_0\big(\frac{\cdot}{\varepsilon}\big) \hat f, \varphi\rangle =0.
\end{equation}

In fact,
\begin{align*}
|\langle  \varphi_0\big(\frac{\cdot}{\varepsilon}\big) \hat f, \varphi\rangle|
&= |\langle   \hat f, \varphi_0\big(\frac{\cdot}{\varepsilon}\big) \varphi\rangle|\\
&=\Big|\Big \langle    f, \Big(\varphi_0\big(\frac{\cdot}{\varepsilon}\big) \varphi\Big)\,\hat{}\,
      \Big\rangle\Big|\\
&=\Big|\Big \langle    f,
\varepsilon^N \hat \varphi_0(\varepsilon\cdot ) * \hat \varphi
  \Big\rangle\Big|\\
&\le \|f\|_{L^{p_0}}
   \cdot   \|\varepsilon^N \hat \varphi_0(\varepsilon\cdot ) * \hat \varphi\|_{L^{p'_0}} \\
&\le \|f\|_{L^{p_0}}
   \cdot   \|\varepsilon^N \hat \varphi_0(\varepsilon\cdot )\|_{L^{p'_0}} \cdot \| \hat \varphi\|_{L^1} \\
&\rightarrow 0.
\end{align*}
Hence (\ref{eq:s:2}) is true.
Consequently,
\[
    \hat f = \lim_{\varepsilon\rightarrow 0} \Big(1 -\varphi_0\big(\frac{\cdot}{\varepsilon}\big) \Big)
    \hat f  \qquad \mbox{in} \quad \mathscr S'.
\]
It follows that
for any $\varphi\in \mathscr S$,
\begin{align}
 \langle \hat f, \varphi\rangle
 &= \lim_{\varepsilon\rightarrow 0} \Big\langle
 \Big(1 -\varphi_0\big(\frac{\cdot}{\varepsilon}\big) \Big)
    \hat f  , \varphi\Big\rangle  \nonumber \\
 &= \lim_{\varepsilon\rightarrow 0} \Big\langle
     \hat f  , \Big(1 -\varphi_0\big(\frac{\cdot}{\varepsilon}\big) \Big)\varphi\Big\rangle
     \nonumber
     \\
 &= \lim_{\varepsilon\rightarrow 0}
    \int_{\bbR^N} \lambda(x) \Big(1 -\varphi_0\big(\frac{x}{\varepsilon}\big) \Big)\varphi(x)\rmd x
   \nonumber  \\
 &= \lim_{\varepsilon\rightarrow 0}
    \int_{|x|>\varepsilon } \lambda(x) \Big(1 -\varphi_0\big(\frac{x}{\varepsilon}\big) \Big)\varphi(x)\rmd x.
    \label{eq:s:5}
\end{align}
Hence $\hat f = \lim_{\varepsilon\rightarrow 0}
  (1 - \varphi_0(x/\varepsilon)) \lambda(x)$
in distributional sense.

Recall that $|x|^{\beta_s}\lambda(x) \in L^{p'_{\lceil s\rceil}}$.
If
$\beta_sp_{\lceil s\rceil}<N$,
we see from H\"older's inequality that
  $\lambda(x) = |x|^{-\beta_s} |x|^{\beta_s}\lambda(x)\in   L^1(B(0,R))$
for any $R>0$.
Hence $\lambda(x)$ is locally integrable.
This completes the proof.
\end{proof}

We are now ready to prove the main results.
\bigskip

\begin{proof}[Proof of Theorem~\ref{thm:convergence:2}]
Fix some $f\in W_s^{\vec p}$.
By Lemma~\ref{Lm:frac:Sobolev},
there is some $h\in L^2$ such that $\hat f(\omega) =  h(\omega) /|\omega|^s$.

Take some function $\varphi$ such that
$\hat \varphi \in C_c^{\infty}$, $0\le \hat\varphi(\omega)\le 1$,
$\hat \varphi(\omega)=1$ for $|\omega|<1$
and $\hat \varphi(\omega)=0$ for $|\omega|>2$.
Set $\hat f_1 = \hat \varphi\cdot \hat f$ and $f_2 = f - f_1$. Then the decomposition (\ref{eq:s:phi})  is true.

First, we consider the case $s<N/2$.
In this case,  we have
\[
  \hat f_1 (\omega) = \frac{1}{|\omega|^s} 1_{\{|\omega|<2\}}(\omega)
   \cdot \varphi(\omega)  h(\omega)   \in L^1.
\]
It follows from the dominated convergence theorem that
$e^{it(-\Delta)^{a/2}}f_1 (x)$ tends to $f_1(x)$ as $t$ tends to $0$.

On the other hand,
since $\hat f_2(\omega) = (1 - \hat \varphi(\omega))\hat f(\omega) = 0 $
for $|\omega| \le 1$,
we have
\[
  \|\hat f_2 \|_{L^2}
 \le \| |\omega|^{s}\hat f_2(\omega)  \|_{L^2}
 \le \| |\omega|^{s}\hat f(\omega)  \|_{L^2}
 <\infty.
\]
Hence   $f_2\in H^s$.
By the hypothesis, we
get
$\lim_{t\rightarrow 0} e^{it(-\Delta)^{a/2}}f_2 (x)=f_2(x)$, a.e.

Next we consider the case $s\ge N/2$ and $a=2$.

Set $\varphi_0=\varphi$.
Since $1- \varphi_0( \omega/\varepsilon)=0$ for $|\omega|<\varepsilon$,
we rewrite $e^{-it\Delta} f$ as
\begin{align}
   e^{-it\Delta} f(x)
   &=\lim_{\varepsilon\rightarrow 0} \frac{1}{(2\pi)^N}
      \int_{|\omega|>\varepsilon} e^{ i(x\cdot \omega + t |\omega|^2 )}
      \Big( 1- \varphi_0\big(\frac{\omega}{\varepsilon}\big)\Big)\hat f(\omega) \rmd \omega
      .  \label{eq:Schr:3}
\end{align}
Recall that $|\omega|^s\hat f(\omega)\in L^{2}$.
Hence $\hat f  \cdot 1_{\{|\omega|>\varepsilon\}}\in L^2$.
Therefore,
the integral in the above equation is well defined.
Let us show that the limit in (\ref{eq:Schr:3})
exists almost everywhere.
Observe that
\begin{align*}
& \int_{|\omega|>\varepsilon} e^{ i(x\cdot \omega + t |\omega|^2 )}
      \Big( 1- \varphi_0\big(\frac{\omega}{\varepsilon}\big)\Big)\hat f(\omega) \rmd \omega\\
&=    \int_{|\omega|>\varepsilon} e^{ i(x\cdot \omega + t |\omega|^2 )}
      \Big( 1- \varphi_0\big(\frac{\omega}{\varepsilon}\big)\Big)\hat \varphi(\omega)\hat f(\omega) \rmd \omega
  + \int_{|\omega|>1} e^{ i(x\cdot \omega + t |\omega|^2 )}
       \hat f_2(\omega) \rmd \omega   .
\end{align*}
By (\ref{eq:s:5}), we have
\begin{align*}
e^{-it\Delta} f_1(x)&=
\lim_{\varepsilon\rightarrow 0} \frac{1}{(2\pi)^N}
\int_{|\omega|>\varepsilon} e^{ i(x\cdot \omega + t |\omega|^2 )}
      \Big( 1- \varphi_0\big(\frac{\omega}{\varepsilon}\big)\Big)\hat \varphi(\omega)\hat f(\omega) \rmd \omega \\
&
  =  \frac{1}{(2\pi)^N}\langle \hat f,     e^{ i(x\cdot \omega + t |\omega|^2 )}\hat \varphi(\omega)\rangle.
\end{align*}
That is, the limit in (\ref{eq:Schr:3}) exists for
almost all $x$.

Observe that
$e^{ i(x\cdot \omega + t |\omega|^2 )}\hat \varphi(\omega)$ tends to $e^{ ix\cdot \omega}
\hat \varphi(\omega)$ in $\mathscr S$ as $t\rightarrow 0$. We get
\begin{align*}
 \lim_{t\rightarrow 0} e^{-it\Delta} f_1(x)
  &=\lim_{t\rightarrow 0}   \frac{1}{(2\pi)^N}\langle \hat f,     e^{ i(x\cdot \omega + t |\omega|^2 )}\hat \varphi(\omega)\rangle \\
  &= \frac{1}{(2\pi)^N}
  \langle \hat f,     e^{ i x\cdot \omega }\hat \varphi(\omega)\rangle  \\
 &=\langle  f,      \varphi(x-\cdot)\rangle   \\
 &=f*\varphi(x).
\end{align*}

As in the previous case, $f_2\in H^s$. By the hypothesis,
  $\lim_{t\rightarrow 0} e^{it(-\Delta)^{a/2}}f_2 (x)=f_2(x)$, a.e.
Hence $\lim_{t\rightarrow 0} e^{it(-\Delta)^{a/2}}f (x)=f(x)$, a.e.
This completes the proof.
\end{proof}

\bigskip
\begin{proof}[Proof of Theorem~\ref{thm:convergence:p}]
Fix some $f\in W_s^{\vec p}$.
By Lemma~\ref{Lm:frac:Sobolev},
there is some $h\in L^{p'_{\lceil s\rceil}}$ such that $\hat f(\omega) =  h(\omega) /|\omega|^{\beta_s}$.

As in the proof of Theorem~\ref{thm:convergence:2},
we apply the decomposition  (\ref{eq:s:phi})
with $a=2$
for $e^{-it\Delta}f$.

First, we deal with $e^{-it\Delta}f_2 (x)$.
Since $\hat f_2(\omega) = (1 - \hat \varphi(\omega))\hat f(\omega) = 0 $
for $|\omega| \le 1$,
we have
\begin{align}\begin{split}\label{eq:Schr:6}
  \|\hat f_2 \|_{L^2}
  &= \| |\omega|^{-\beta_s} \cdot
   1_{\{|\omega|>1\}}\cdot |\omega|^{\beta_s}\hat f(\omega)  \|_{L^2} \\
& \le\| |\omega|^{-\beta_s} \cdot 1_{\{|\omega|>1\}} \|_{L^r}
 \cdot\| |\omega|^{\beta_s}\hat f(\omega)  \|_{L^{p'_{\lceil s\rceil}}}
 <\infty,
\end{split}\end{align}
where $r>1$ satisfying $1/r    = 1/2-1/p'_{\lceil s\rceil} =
 1/p_{\lceil s\rceil}-1/2 <\beta_s/N$.
More precisely, since
\[
\frac{1}{p_{\lceil s\rceil}} - \frac{1}{2} + \frac{1}{2(N+1)}
=
\frac{1}{p_{\lceil s\rceil}} - \frac{N}{2(N+1)}
   <\frac{\beta_s}{N}.
\]
We have
\[
  \frac{1}{r}
   <\frac{\beta_s-N/(2(N+1))}{N}.
\]
Consequently, there is some $\tau>  N/(2(N+1))$ such that
\[
  \frac{1}{r} < \frac{\beta_s-\tau}{N}.
\]
Applying H\"older's inequality yields
\[
  \| |\omega|^{\tau} \hat f(\omega)\|_{L^2}
   \le     \Big\||\omega|^{-(\beta-\tau)} 1_{\{|\omega|>1\}}\Big\|_{L^r}
    \cdot \Big\||\omega|^{\beta_s}
       \hat f(\omega)\Big\|_{L^{p'_{\lceil s\rceil}}}<\infty.
\]
Thus $f_2\in H^{\tau}$.

Recall that it was
proved in  \cite{Carleson1980,DuGuthLi2017,DuZhang2019}
that $\lim_{t\rightarrow 0} e^{-it\Delta}f (x)=f(x)$, a.e.,
for all $f\in H^{\tau}$ with $\tau > N/(2(N+1))$. Hence
\[
  \lim_{t\rightarrow 0} e^{-it\Delta}f_2 (x)=f_2(x),\quad  a.e.
\]

Next we deal with $ e^{-it\Delta} f_1$.
If  $\beta_s<N/ p_{\lceil s\rceil}$, then we have $\hat f_1 \in L^1$.
It follows from the dominated convergence theorem that
$e^{-it\Delta}f_1 (x)$ tends to $f_1(x)$ as $t$ tends to $0$.

If $\beta_s \ge N/p_{\lceil s\rceil}$,
employing the same arguments as in the case  $s\ge N/2$
in the proof of Theorem~\ref{thm:convergence:2},
we obtain
\begin{align*}
 \lim_{t\rightarrow 0} e^{-it\Delta} f_1(x)
  &=f*\varphi(x).
\end{align*}
Hence $\lim_{t\rightarrow 0} e^{-it\Delta} f(x) = f(x)$,  a.e.
This completes the proof.
\end{proof}


 \end{document}